\documentclass[10pt,reqno]{amsart}
\usepackage{graphicx,epsfig}
\usepackage{epstopdf}
\usepackage{pdfpages}
\usepackage{amssymb}
\usepackage{empheq}
\usepackage{cases}
\usepackage{amsthm,amsmath}
\usepackage{caption,lipsum}
\usepackage{stmaryrd}
\usepackage{tabularx}
\usepackage{color}
\usepackage{empheq,float}
\DeclareGraphicsExtensions{.eps}
\newtheorem{theorem}{Theorem}[section]

\newtheorem{corollary}[theorem]{Corollary}
\theoremstyle{definition}

\theoremstyle{remark}
\newtheorem{remark}[theorem]{Remark}

\newcommand{\fe}{\mathrm{e}}

\newcommand{\bR}{{\mathbb R}}
\newcommand{\bT}{{\mathbb T}}

\newcommand{\bC}{{\mathbb C}}
\newcommand{\bZ}{{\mathbb Z}}

\newcommand{\p}{\partial}

\newcommand{\be}{\begin{equation}}
\newcommand{\ee}{\end{equation}}
\newcommand{\bes}{\begin{equation*}}
\newcommand{\ees}{\end{equation*}}
\numberwithin{equation}{section}
\begin{document}

\title{Low-regularity integrators for nonlinear Dirac equations}

\author[K. Schratz]{Katharina Schratz}
\address{\hspace*{-12pt}K.~Schratz: Fakult\"{a}t f\"{u}r Mathematik, Karlsruhe Institute of Technology,
Englerstr.2, 76131 Karlsruhe, Germany}
\email{katharina.schratz@kit.edu}

\author[Y. Wang]{Yan Wang}
\address{\hspace*{-12pt}Y.~Wang: School of Mathematics and Statistics, Central China
Normal University, 430079 Wuhan, China}
\email{wang.yan@mail.ccnu.edu.cn}

\author[X. Zhao]{Xiaofei Zhao}
\address{\hspace*{-12pt}X.~Zhao: School of Mathematics and Statistics, Wuhan University, 430072 Wuhan, China}
\email{matzhxf@whu.edu.cn}

\date{}

\dedicatory{}
\begin{abstract}
In this work, we consider the numerical integration of the nonlinear Dirac equation and the Dirac-Poisson system (NDEs) under rough initial data. We propose a ultra low-regularity integrator (ULI) for solving the NDEs which enables optimal first-order time convergence in $H^r$  for solutions in $H^{r}$, i.e., without requiring any additional regularity on the solution. In contrast to classical methods, ULI overcomes the numerical loss of derivatives and is therefore more efficient and accurate for approximating low regular solutions.
Convergence theorems and the extension of ULI to second order are established.  Numerical experiments confirm the theoretical results and underline the favourable error behaviour of the new method at low regularity compared to classical integration schemes.
\end{abstract}

\maketitle

\section{Introduction}
Numerical integrators for solving the semi-linear dispersive equation
$$\partial_t u=\mathcal{L}u+\mathcal{N}(u),\quad t>0,$$
usually require smoothness of the solution $u$, i.e., the boundedness of $\mathcal{L}^p u$ in some Sobolev space, where $\mathcal{L}$ denotes  a  skew-adjoint linear differential operator and $\mathcal{N}(u)$ denotes some nonlinear function \cite{Tao}. In particular, traditional methods, such as explicit and implicit Runge-Kutta methods, splitting schemes and exponential integrators \cite{Hochbruck}, can only reach their optimal convergence rate $O((\Delta t)^m)$ in $H^r$ ($r> \frac{1}{2}$) for solutions in $H^{r+\delta}$ for some $\delta = \delta(m, \mathcal{L})>0$ depending on the order of spatial differentiation in $\mathcal{L}$ and $m$.
 This additional regularity requirement on the solution (i.e., $\delta>0$) is indeed introduced by the numerical approximation and necessary for (optimal) convergence. Recently,  for certain problems a new class of integrators could be constructed which allow optimal convergence without any loss of numerical derivatives in the time discretization (i.e., $\delta = 0$), see for instance \cite{lownls} for the one-dimensional quadratic nonlinear Schr\"{o}dinger equation.
In this work, we identify the nonlinear Dirac equations, i.e., a class of nonlinear dispersive equations with cubic nonlinearities, for which such low-regularity integrators can be designed.
As important models in particle physics and relativistic quantum mechanics, the Dirac-type equations have been widely considered in studies of two-dimensional materials \cite{FDDirac,2dmaterial}, quantum field theory \cite{Thirring,QFT} and Bose-Einstein condensates \cite{BEC}. We shall consider in this paper the following one-dimensional nonlinear Dirac equation (NDE) \cite{DiracYan,DiracZhao} as the model problem:
\begin{subequations}\label{dirac}
\begin{align}
&i\partial_{t}\Phi=-i\alpha\partial_x\Phi+\beta\Phi +V\Phi+ F(\Phi),\quad t>0,\ x\in\bR,\\
&\Phi(0,x)=\Phi_0(x),\quad x\in\bR, \label{dirac 2}
\end{align}
\end{subequations}
where $\Phi=\Phi(t,x)=(\phi_1(t,x), \phi_2(t,x))^T:[0,\infty)\times\bR\to\bC^2$ is the unknown spinorfield, $\Phi_0(x):\bR\to\bC^2$ is the initial data, $V$ denotes the electric potential which is either given as an external real-valued function $V=V_e(x)$ or a time-dependent function $V=V(t,x)$ determined by a self-consistent Poisson equation \cite{FDDirac}
$$-\partial_{xx} V=|\Phi|^2,\quad t\geq0,\ x\in\bR\quad\mbox{with}\quad \int_{\bR}Vdx\equiv0,
\quad t\geq0,$$
$F(\Phi) = \lambda (\Phi^*\beta\Phi)\beta\Phi=\lambda(|\phi_1|^2-|\phi_2|^2)\beta\Phi$ denotes the cubic Thirring-type nonlinearity with a given parameter $\lambda\in\bR$ denoting the strength of the nonlinear interaction \cite{Thirring}, where $\Phi^* = \bar{\Phi}^T$ denotes the complex conjugate transpose of $\Phi$, and $$
\alpha=
\begin{pmatrix}
0 & 1 \\
1&0
\end{pmatrix},\quad
\beta=
\begin{pmatrix}
1 & 0 \\
0&-1
\end{pmatrix},$$
are the Pauli matrices.

Equation (\ref{dirac}) occurs as a core part of many related models, such as the Gross-Neveu system \cite{lowDKG5}, the Dirac-Klein-Gordon system \cite{lowDKG2,KGD}, the Maxwell-Dirac system \cite{lowDKG3,DiracMaxwell} and the Chern-Simons-Dirac system \cite{lowCSD}. Due to these applications, the numerical solutions of the NDE (\ref{dirac}) are widely interested in computational physics \cite{DMBao, FDDirac,SplittingDirac2,DiracMaxwell,DiracZhao,Splitadd1,FDadd1}, especially the bound states, the solitary waves and the interaction dynamics of solitons \cite{1Dapp0, 1Dapp1, Shao}. These existing works in computational physics or applied mathematics mainly assume that the initial data of NDE is smooth enough for theoretical and/or  numerical studies. While in general physical situations or real applications with unavoidable noise, the initial input in (\ref{dirac 2}) for the NDE may not be a smooth function or a function with high regularity. For such initial data, mathematical analysis of NDE has been intensively carried out in the literature. The global existence of the NDE (\ref{dirac}) was firstly established in \cite{H1result} for $H^1$ initial data, i.e., $\Phi_0\in H^1(\bR)$. Later, many efforts \cite{lowDKG0,lowdirac0,lowdirac1,Wellpose0,Machihara,lowdirac2} were made to obtain the local and/or global well-posedness of (\ref{dirac}) by allowing initial data with lower regularities. Among them, Selberg and Tesfahun proved the local well-posedness for initial data in almost critical space $H^r$ with any $r>0$ \cite{lowdirac2}, and Candy extended their result to the critical case for any $r\geq0$ \cite{lowdirac1}. We refer interested readers to \cite{lowdiracbook} for a systematical review and some refined Strichartz estimates. Corresponding studies have also been made in high space dimensions case \cite{lowdirachd,lowdirac3d,lowdirac2D} and in coupling case with Klein-Gordon equation or  Maxwell equations et al. \cite{lowDKG1,lowCSD,lowDKG2,lowDKG3,lowDKG4,lowDKG5,lowDKG6,lowDKG7}.

However, the numerical methods proposed for solving the NDE in the literature so far all lead to numerical loss of regularity. More precisely, for the numerical solutions to have a $m$-th ($m\geq1$) order of convergence in time in the Sobolev space  $H^r$, the finite difference time domain methods in \cite{FDDirac,FDadd2,FDadd1} and the Runge-Kutta discontinuous Galerkin methods in \cite{RKdirac} both require initial data at least in $H^{r+m+1}$, and the splitting spectral methods in \cite{DMBao,SplittingDirac2,DiracMaxwell,Splitadd1} and the exponential integrator spectral methods in \cite{JSC,SplittingDirac1,EIadd1} require initial data in $H^{r+m}$. We refer to \cite{Shao} for a detailed numerical comparison of these methods with smooth initial data. For initial data with less regularity than the required, these numerical schemes will fail to reach their optimal convergence rates, suffering from sever order reduction and hence become less efficient and accurate (see also Section \ref{sec:result} for a numerical illustration of this phenomenon). We shall summarize and emphasize the limitations of these methods in the paper by writing down their rigorous error estimates with critical regularity requirements.

To enable efficient numerical integrations of NDE with rough initial data, we shall present a class of integrators in the spirit of low-regularity methods \cite{Kdv,lownls,lowNLS2}, for solving (\ref{dirac}). The methods are derived under the framework of the nested Picard iterative integrators (NPI) \cite{DiracYanSINUM}, while we modify the evolution operator, instead of using the one with ``free Dirac operator'' \cite{Mauser, DiracYanSINUM}. This allows for a decomposition in eigenspaces which ensures exact integrations of terms in the Fourier frequency space in an explicit and efficient way.

 The main novelty  of the paper is to introduce a new class of  \emph{ultra low-regularity integrator (ULI)}  integrators for Dirac equations. This new class will offer optimal first-order time convergence in  $H^r$ for any  $ H^{r}$-data (for $r>\frac12$), i.e., no auxiliary smoothness is needed from the solution at all if one does not take the spatial discretization into account. Therefore, we say that the scheme is a \emph{ultra low-regularity integrator}. In the literature, such ULI methods could only be designed so far for the quadratic Schr\"{o}dinger equation \cite{lownls}. On the other hand, for the cubic nonlinear Schr\"{o}dinger equations or the KdV equation, some addition regularity requirement has to be imposed \cite{Kdv,lownls3,lownls,lowNLS2}.

Based on the first-order ULI method, we build another scheme that offers optimal second-order convergence in $H^r$ for $ H^{r+1}$-data ($r \geq 0$). In particular, second-order convergence in time holds in $L^2$ for solutions in $H^1$.  Up to our knowledge no scheme which allows for such a generous convergence has been proposed in literature so far and compared to the classical Strang splitting scheme, one space derivative is saved by our new second-order method. Note that the requirement that the solution is in $H^1$ can be seen naturally as the spatial convergence requires some smoothness of the solution. We focus on first- and second-order integrators.  Our new  framework can, however, be (in principle) generalized to arbitrary high-order low-regularity integrators. Rigorous convergence results are established and numerical results are presented to underline the performance of the proposed methods in comparison to classical schemes.

The rest of this paper is organized as follows. In Section \ref{sec:2}, we review some standard numerical methods for solving the NDE (\ref{dirac}). Thereby we pay  particular attention on the required regularities. In Section \ref{sec:3}, we present the first-order ULI scheme and its convergence result. Extensions of ULI to the second order are made in Section \ref{sec:4}. Numerical results are presented in Section \ref{sec:result} and some conclusions are drawn in Section \ref{sec:con}.

\section{Convergence of standard methods}\label{sec:2}
In this section, we shall review some of the popular integrators including finite difference methods, exponential integrators and splitting methods for NDE (\ref{dirac}). For simplicity of presentation, we will assume that $V = V_e(x)$ is a given external smooth function in this section and we will focus on the time integrations. As for spatial discretization, finite difference methods \cite{FDDirac}, discontinuous Galerkin methods \cite{RKdirac} and spectral methods \cite{DiracYan,DiracZhao} can all be applied. To simplify numerical implementations, we truncate the whole space problem (\ref{dirac}) onto the torus $\bT=\bR/(2\pi)$:
\begin{subequations}\label{dirac trun}
\begin{align}
&i\partial_{t}\Phi=-i\alpha\partial_x\Phi+\beta\Phi +V\Phi+F(\Phi),\quad t>0,\ x\in\bT,\label{dirac trun a}\\
&\Phi(0,x)=\Phi_0(x),\quad x\in\bT,
\end{align}
\end{subequations}
and impose periodic boundary conditions so that Fourier spectral/pseudospectral method can be easily applied. To simplify the notations, we shall omit the space variable in the following, i.e., we will write $\Phi(t)=\Phi(t,x)$. We denote by $\tau=\Delta t>0$ the time step, $t_n=n\tau$ as the time grids, and define $\Phi^n=(\phi_1^n, \phi_2^n)^T\approx\Phi(t_n)$ as the numerical solution. We write  $\|\cdot\|_{r}:=\|\cdot\|_{H^r(\bT)}$ for the standard Sobolev norm, and as for a vector field $\Psi=(\psi_1,\psi_2)^T$ on $\bT$, we set $\|\Psi\|_{r}:=\sqrt{\|\psi_1\|_r^2+\|\psi_2\|_r^2}$.

For each of the reviewed standard methods in the following, we write down their convergence theorems to address their critical regularity requirements for convergence. The proofs of the theorems are given in Appendix \ref{appendix}.

\subsection{Finite difference methods} As the most traditional numerical discretization, finite difference methods have been widely applied for solving the NDE \cite{SplittingDirac1,FDadd2,FDadd1,RKdirac} and coupled systems, such as the  Dirac-Poisson  and Klein-Gordon-Dirac systems \cite{FDDirac,KGD}. Here, we present two semi-implicit finite difference integrators which are free from CFL conditions.

A first-order semi-implicit finite difference integrator (FD1) for the NDE (\ref{dirac trun}) reads:
\begin{subequations}\label{SIFD1}
\begin{align}
&i\frac{\Phi^{n+1}-\Phi^{n}}{\tau}=-i\alpha\partial_x\Phi^{n+1}+\beta\Phi^n +V\Phi^n+F(\Phi^n),\quad n\geq0,\\
&\Phi^0=\Phi_0,
\end{align}
\end{subequations}
and a second-order semi-implicit finite difference integrator (FD2)  reads:
\begin{subequations}\label{SIFD2}
\begin{align}
&i\frac{\Phi^{n+1}-\Phi^{n-1}}{2\tau}=-i\alpha\partial_x\frac{\Phi^{n+1}+\Phi^{n-1}}{2}
+\beta\Phi^n +V\Phi^n+F(\Phi^n),\quad n\geq0,\\
&\Phi^0=\Phi_0,\quad \Phi^1=\Phi_0-i
\tau\left[-i\alpha\partial_x\Phi_0+\beta\Phi_0 +V\Phi_0+F(\Phi_0)\right],
\end{align}
\end{subequations}
where $\Phi^1$ in the second-order scheme is obtained by a first-order Taylor expansion of equation (\ref{dirac trun a}) at $t=0$.

\begin{theorem}\label{thm FD}(Convergence of finite difference methods) Let $\Phi^n$ denote the numerical solution of the  FD1 scheme (\ref{SIFD1}) for solving the NDE (\ref{dirac trun}). Let $r>\frac{1}{2}$ and $\partial_t^m\Phi\in L^\infty((0,T);H^{r+2-m})$ where $m=0,1,2$, for some $T>0$. Then there exist constants $\tau_0,\,C>0$ depending on $ \Vert \Phi\Vert_{L^\infty((0,T);H^{r+2})}$, $ \Vert \partial_t\Phi\Vert_{L^\infty((0,T);H^{r+1})}$, $ \Vert \partial_{tt}\Phi\Vert_{L^\infty((0,T);H^{r})}$ and $T$, such that for all $0<\tau\leq\tau_0$ and $0\leq t_n\leq T$, we have
 $$\|\Phi(t_n)-\Phi^n\|_{r}\leq C\tau.$$
 Further under assumption $\partial_t^m\Phi\in L^\infty((0,T);H^{r+3-m})$ where
 $m=0,\ldots,3$,  for $\Phi^n$ from the FD2 scheme (\ref{SIFD2}), there exist constants $\tau_0,\,C>0$ depending on $ \Vert \Phi\Vert_{L^\infty((0,T);H^{r+3})}$, $ \Vert \partial_t\Phi\Vert_{L^\infty((0,T);H^{r+2})}$, $ \Vert \partial_{tt}\Phi\Vert_{L^\infty((0,T);H^{r+1})}$, $ \Vert \partial_{ttt}\Phi\Vert_{L^\infty((0,T);H^{r})}$ and $T$, such that for all $0<\tau\leq\tau_0$ and $0\leq t_n\leq T$, we have
 $$\|\Phi(t_n)-\Phi^n\|_{r}\leq C\tau^2.$$
\end{theorem}

\subsection{Classical exponential integrators}\label{sec:Cexp} Exponential integrators \cite{Hochbruck} have been intensively developed and analyzed for various  evolution equations, particularly for equations involving a stiff linear term. They have been recently considered to solve NDEs in the nonrelativistic limit regime in \cite{JSC, SplittingDirac1,EIadd1}. In the following we recall their construction in case of the Dirac equation.

By writing the NDE (\ref{dirac trun}) for $t\geq t_n$ under the Duhamel's formula, we have
\begin{equation}\label{duhammel0}
\Phi(t_{n+1})=\fe^{-i\tau\mathcal{T}}\Phi(t_n)-i\int_0^\tau\fe^{-i(\tau-s)\mathcal{T}}G(\Phi(t_n+s))\, ds,\quad n\geq0,
\end{equation}
with
\bes
    \mathcal{T}: = -i\alpha\p_x + \beta, \qquad G(\Phi(t_n+s)) = V\Phi(t_n+s) + F(\Phi(t_n+s)).
\ees
Here  $\mathcal{T}$ is known as the free Dirac operator \cite{Mauser}.
By approximating $\Phi(t_n+s)$ in the integrant with $\Phi(t_n)$,
the first-order classical Gautschi-type exponential integrator (EI1) \cite{Gautschi} reads
\be\label{EI1}
   \Phi^{n+1}=\fe^{-i\tau\mathcal{T}}\Phi^n-i\tau\varphi_1(-i\tau\mathcal{T})G^n,\  n\geq0,
\ee
with $G^n:=G(\Phi^n)$ and
\begin{equation}\label{varphi1}\varphi_1(z)=\frac{\fe^{z}-1}{z}.\end{equation} On the other hand, by extrapolation  $G(\Phi(t_n+s))\approx G(\Phi(t_n))+\frac{s}{\tau}(G(\Phi(t_n))-G(\Phi(t_{n-1})))$ as in \cite{SplittingDirac1}, the second-order Gautschi-type exponential integrator (EI2)  reads
 \begin{align}\label{EI2}
   \Phi^{n+1}=\fe^{-i\tau\mathcal{T}}\Phi^n-i\tau\varphi_1(-i\tau\mathcal{T})
   G^n-i\tau\varphi_2(-i\tau\mathcal{T})(G^n-G^{n-1}),\quad n\geq1,
   \end{align}
with $\varphi_2(z)=\frac{\fe^{z}-z-1}{z^2}$.

\begin{theorem}\label{thm EI}(Convergence of exponential integrators) Let $\Phi^n$ denote the numerical solution of the EI1 scheme (\ref{EI1}) for solving the NDE (\ref{dirac trun}). Let $r>\frac{1}{2}$ and $\partial_t^m\Phi\in L^\infty((0,T);H^{r+1-m})$ where $m=0,1$, for some $T>0$. Then there exist constants $\tau_0,\,C>0$ depending on $ \Vert \Phi\Vert_{L^\infty((0,T);H^{r+1})}$, $ \Vert \partial_t\Phi\Vert_{L^\infty((0,T);H^{r})}$  and $T$, such that for all $0<\tau\leq\tau_0$ and $0\leq t_n\leq T$, we have
 $$\|\Phi(t_n)-\Phi^n\|_{r}\leq C\tau.$$
 Further under assumption $\partial_t^m\Phi\in L^\infty((0,T);H^{r+2-m})$ where $m=0,1,2$,
then for $\Phi^n$ from the EI2 scheme (\ref{EI2}), there exist constants $\tau_0,\,C>0$ depending on $ \Vert \Phi\Vert_{L^\infty((0,T);H^{r+2})}$, $ \Vert \partial_t\Phi\Vert_{L^\infty((0,T);H^{r+1})}$, $ \Vert \partial_{tt}\Phi\Vert_{L^\infty((0,T);H^{r})}$ and $T$, such that for all $0<\tau\leq\tau_0$ and $0\leq t_n\leq T$, we have
 $$\|\Phi(t_n)-\Phi^n\|_{r}\leq C\tau^2.$$
\end{theorem}

\subsection{Splitting methods} The time splitting or operator splitting method can be considered as one of the most popular numerical techniques for solving evolutions equations of first order in time \cite{Splitting}.
They have been proposed in the literature for the NDEs \cite{SplittingDirac1,SplittingDirac2,Splitadd1} and for Maxwell-Dirac system \cite{DMBao,DiracMaxwell}. The idea of the time splitting methods is to split (\ref{dirac trun}) into the  two systems
\bes
\Phi_{kin}(t):\quad i\partial_t\Phi=-i\alpha\partial_x\Phi
+\beta\Phi,\quad t>0,\ x\in\bT,
\ees
and
\bes
\Phi_{pon}(t):\quad i\partial_t\Phi=V\Phi+F(\Phi),\quad t>0,\ x\in\bT.
\ees
Both of the resulting flows can be exactly integrated in time, i.e.,
$$\Phi_{kin}(t):\ \Phi(t)=\fe^{-it\mathcal{T}}\Phi(0),\quad
\Phi_{pon}(t):\ \Phi(t)=\fe^{-it\left(V\cdot Id+\lambda(|\phi_1|^2-|\phi_2|^2)\beta\right)}\Phi(0),$$
where $Id$ is the $2\times2$ identity matrix.
The two flows are then combined as $\Phi(\tau)\approx \Phi_{kin}\left(\tau\right)\circ\Phi_{pon}\left(\tau
\right)$ for the Lie splitting scheme:
\begin{subequations}\label{TSFP1}
\begin{align}
&\Phi^n_-=\fe^{-i\tau\left(V\cdot Id+\lambda(|\phi_1^n|^2-|\phi_2^n|^2)\beta\right)}\Phi^n,\quad n\geq0,\label{TSFP1a}\\
&\Phi^{n+1}=\fe^{-i\tau\mathcal{T}}\Phi^n_-,\label{TSFP1b}
\end{align}
\end{subequations}
or combined as $\Phi(\tau)\approx \Phi_{pot}\left(\frac{\tau}{2}\right)\circ\Phi_{kin}\left(\tau
\right)\circ\Phi_{pot}\left(\frac{\tau}{2}\right)$ for the Strang splitting scheme:
\begin{subequations}\label{TSFP2}
\begin{align}
&\Phi^n_-=\fe^{-i\frac{\tau}{2}\left(V\cdot Id+\lambda(|\phi_1^n|^2-|\phi_2^n|^2)\beta\right)}\Phi^n,\quad n\geq0,\\
&\Phi^{n}_+=\fe^{-i\tau\mathcal{T}}\Phi^n_-,\\
&\Phi^{n+1}=\fe^{-i\frac{\tau}{2}\left(V\cdot Id+\lambda
(|\phi_{+,1}^n|^2-|\phi_{+,2}^n|^2)\beta\right)}\Phi^n_+,
\end{align}
\end{subequations}
with $\Phi_+^n = (\phi_{+,1}^n, ~\phi_{-,1}^n)^T, \Phi^0=\Phi_0$.

\begin{theorem}\label{thm SP}(Convergence of splitting methods)  Let $\Phi^n$ denote the numerical solution of the  Lie splitting scheme (\ref{TSFP1}) for solving the NDE (\ref{dirac trun}). Let $r>\frac{1}{2}$ and $\Phi\in L^\infty((0,T);H^{r+1})$ for some $T>0$. Then there exist constants $\tau_0,\,C>0$ depending on $ \Vert \Phi\Vert_{L^\infty((0,T);H^{r+1})}$ and $T$, such that for all $0<\tau\leq\tau_0$ and $0\leq t_n\leq T$, we have
 $$\|\Phi(t_n)-\Phi^n\|_{r}\leq C\tau.$$
 Further under assumption $\Phi\in L^\infty((0,T);H^{r+2})$,
then for $\Phi^n$ from the Strang splitting scheme (\ref{TSFP2}), there exist constants $\tau_0,\,C>0$ depending on $ \Vert \Phi\Vert_{L^\infty((0,T);H^{r+2})}$ and $T$, such that for all $0<\tau\leq\tau_0$ and $0\leq t_n\leq T$, we have
 $$\|\Phi(t_n)-\Phi^n\|_{r}\leq C\tau^2.$$
\end{theorem}


\section{First-order ultra low-regularity integrators}\label{sec:3}
In this section, we derive the ultra low-regularity integrators of first-order accuracy for solving the NDEs and present their convergence theorems. We shall consider in a sequel the case of an external electrical potential and the case of  Dirac-Poisson system.

\subsection{NDE with external field}
We firstly consider the NDE \eqref{dirac trun} with a given function $V=V_e(x)$, under the same notations as in Section \ref{sec:2}. Similar to exponential integrators in Section \ref{sec:Cexp}, our method is also constructed based on the integral form of the NDE. However, we apply Duhamel's formula for \eqref{dirac trun} for $n=0,1,\ldots,\,s\geq0$ in an alternative way as follows
\begin{align}\label{dumammel}
  \Phi(t_n+s)=\fe^{-s\alpha\partial_x}\Phi(t_n)-i\int_0^s\fe^{-(s-\rho)\alpha\partial_x}
 \left[\beta\Phi(t_n+\rho) +V\Phi(t_n+\rho) + F(\Phi(t_n+\rho))\right] d\rho.
\end{align}
Here, we use the evolution operator $\fe^{-s \alpha \p_x}$ instead of $\fe^{-i\mathcal{T}s}$ in \eqref{duhammel0}, which is crucial to the success of the method. It is known that the `free Dirac operator' $\mathcal{T}$ is diagonalizable in Fourier space and it can be decomposed as \cite{Mauser}
\bes
    \mathcal{T} = \sqrt{Id-\partial_{xx}}~\Pi_+^{\mathcal{T}} - \sqrt{Id-\partial_{xx}}~\Pi_-^{\mathcal{T}},
\ees
with the projectors $\Pi_\pm^{\mathcal{T}}$ defined as
\bes
    \Pi_\pm^{\mathcal{T}} = \frac{1}{2}\left[Id \pm (Id-\partial_{xx})^{-1/2}\mathcal{T} \right].
\ees
As can be seen that the spatial differentiation operator is involved nonlinearly in the above decomposition. This works well in the numerical studies for the nonrelativistic limit regime \cite{DiracYan,DiracYanSINUM}, but here it makes things difficult in designing low-regularity methods.
In contrast, the operator $\alpha \p_x$ can be decomposed as
\be
    \alpha \p_x = \p_x \Pi_+ - \p_x \Pi_-
\ee
with
\bes
    \Pi_+ = \frac{1}{2} \begin{pmatrix}
1 & 1 \\
1 & 1
\end{pmatrix}, \qquad \Pi_- = \frac{1}{2} \begin{pmatrix}
1 & -1 \\
-1 & 1
\end{pmatrix},
\ees
where $\p_x$ is linearly involved. It can be verified that $\Pi_+ + \Pi_- = I_2, ~ \Pi_+\Pi_- = \Pi_-\Pi_+ = \mathbf{0}, ~ \Pi_\pm^2 = \Pi_\pm$, and
\bes
    (\alpha\p_x)^k = (\p_x)^k \Pi_+ + (-\p_x)^k \Pi_-, \qquad k \in \mathbf{N},
\ees
which implies the following relation
\be\label{eq:3-3}
    \fe^{s\alpha\p_x} = \fe^{s\p_x} \Pi_+ + \fe^{-s\p_x}\Pi_-.
\ee
Therefore, a nested Picard iteration based on (\ref{dumammel}) can be effectively computed in analogous to \cite{DiracYanSINUM} as follows.

Letting $s=\rho$ for  $0\leq\rho\leq \tau$ in (\ref{dumammel}), we find
\begin{equation}\label{app1}
\Phi(t_n+\rho) = \fe^{-\rho\alpha\partial_x}\Phi(t_n) + O(\tau).
\end{equation}
Here (and in the following) $O(\tau)$ denotes a remainder of order $\tau$ in time which does not require any additional regularity of the solution, i.e.,
\[
\Phi(t_n+\rho) = \fe^{-\rho\alpha\partial_x}\Phi(t_n) + O(\tau)\quad \text{ if } \quad \Vert \Phi(t_n+\rho) - \fe^{-\rho\alpha\partial_x}\Phi(t_n) \Vert_r \leq c \tau
\]
with $c$ depending on $t_n$ and $\sup_{0 \leq \rho \leq \tau} \Vert \Phi (t_n+\rho)\Vert_r$.

Plugging (\ref{app1}) into  (\ref{dumammel}) and setting $s=\tau$, we get
\be\label{sec3 eq1}
    \Phi(t_{n+1}) = \fe^{-\tau \alpha \p_x}\Phi(t_n) - i \fe^{-\tau\alpha\p_x} \left[I_1(t_n) + I_2(t_n) + I_3(t_n) \right] + O(\tau^2),
\ee
where
\begin{equation}\label{I13}
\begin{split}
 &I_1(t_n):=\int_0^\tau\fe^{\rho\alpha\partial_x}\beta\fe^{-\rho\alpha\partial_x}\Phi(t_n) \, d\rho,\quad
 I_2(t_n):=\int_0^\tau\fe^{\rho\alpha\partial_x}V\fe^{-\rho\alpha\partial_x}\Phi(t_n)\,  d\rho,\\
 &I_3(t_n):=\int_0^\tau \fe^{\rho\alpha\partial_x}F(\fe^{-\rho \alpha \p_x} \Phi(t_n))\, d\rho.\end{split}
\end{equation}

For $I_1(t_n)$, by using the relation \eqref{eq:3-3}, we have
\bes
\begin{split}
    \fe^{\rho\alpha\partial_x}\beta\fe^{-\rho\alpha\partial_x}\Phi(t_n) & = (\fe^{\rho\p_x} \Pi_+ + \fe^{-\rho \p_x} \Pi_-)\beta (\fe^{-\rho\p_x} \Pi_+ + \fe^{\rho \p_x} \Pi_-) \Phi(t_n)\\
    & = (\fe^{2\rho \p_x}\beta\Pi_-  + \fe^{-2\rho \p_x}\beta\Pi_+)\Phi(t_n),
\end{split}
\ees
since $\Pi_\pm \beta\Pi_\pm = \mathbf{0}$ and $\Pi_\pm\beta\Pi_\mp = \beta\Pi_\mp$. Therefore, we have the exact integration:
\be\label{eq:I1}
I_1(t_n)=\tau\left[ \varphi_1(2\tau\p_x)\beta\Pi_- + \varphi_1(-2\tau\p_x)\beta\Pi_+\right] \Phi(t_n),
\ee
where $\varphi_1(\cdot)$ is defined in (\ref{varphi1}).

For $I_2(t_n)$, we can analogously write the integrand as
\bes
    \fe^{\rho\alpha\partial_x}V\fe^{-\rho\alpha\partial_x}\Phi(t_n) = \left( \fe^{\rho\p_x} V \fe^{-\rho\p_x}\Pi_+ + \fe^{-\rho\p_x} V \fe^{\rho\p_x}\Pi_- \right)\Phi(t_n).
\ees
What is interesting is that the integration of the above function can be performed in an exact and explicit way in the physical space. For some general function $\psi=\psi(x):\bT\to\bC$, we have
\begin{align*}
 &\int_0^\tau
 \fe^{\rho\partial_x}V\fe^{-\rho\partial_x}\psi d\rho
 =\int_0^\tau \sum_{l\in\bZ}
 \sum_{\begin{array}{cc}&l_1,l_2\in\bZ\\&l_1+l_2=l\end{array}}
\fe^{ilx}\fe^{il\rho}\widehat{V}_{l_1}\fe^{-il_2\rho}\widehat{\psi}_{l_2}d\rho\\
&= \sum_{l\in\bZ}
 \sum_{\begin{array}{cc}&l_1,l_2\in\bZ\\&l_1+l_2=l\end{array}}
\fe^{ilx}\int_0^\tau
\fe^{i(l-l_2)\rho}d\rho\widehat{V}_{l_1}\widehat{\psi}_{l_2}
=\sum_{l\in\bZ}
 \sum_{\begin{array}{cc}&l_1,l_2\in\bZ\\&l_1+l_2=l\end{array}}
 \fe^{ilx}\tau\varphi_1(i\tau l_1)\widehat{V}_{l_1}\widehat{\psi}_{l_2}\\
&=\tau(\varphi_1(\tau\partial_x)V )\psi,
\end{align*}
where $\widehat{\psi}_l$ denotes the Fourier coefficients of $\psi$,
and similarly
$$\int_0^\tau
 \fe^{-\rho\partial_x}V\fe^{\rho\partial_x}\psi d\rho
 =\tau(\varphi_1(-\tau\partial_x)V )\psi.$$
Hence, we have the exact integration for $I_2(t_n)$ as
\be\label{eq:I2}
    I_2(t_n) = \tau \left[(\varphi_1(\tau\p_x) V) \Pi_+ + (\varphi_1(-\tau\p_x) V)\Pi_-\right]\Phi(t_n).
\ee
This exact integration is expressed explicitly in the physical space so that in practice it can be obtained efficiently by means of fast Fourier transform, which is crucial for the success of our method.

For the integration of the nonlinear term in $I_3(t_n)$, firstly we have
\bes
\begin{split}
    g : = & ~ (\fe^{-\rho \alpha \p_x}\Phi(t_n))^* \beta (\fe^{-\rho \alpha \p_x}\Phi(t_n))\\
    = &~ \left[\fe^{-\rho\p_x}(\Phi(t_n))^*\fe^{\rho\p_x}\beta\Pi_- + \fe^{\rho\p_x}(\Phi(t_n))^*\fe^{-\rho\p_x}\beta\Pi_+\right]\Phi(t_n)\\
    = & ~ \frac{1}{2}\left[(\fe^{\rho\p_x}\phi_-(t_n))(\fe^{-\rho\p_x} \overline{\phi_+}(t_n))+ (\fe^{\rho\p_x} \overline{\phi_-} (t_n))(\fe^{-\rho\p_x}\phi_+(t_n)) \right],
\end{split}
\ees
where $$\phi_\pm(t_n) := \phi_1(t_n) \pm \phi_2(t_n),\quad \Phi(t_n)=\left(
\phi_1(t_n), \phi_2(t_n)\right)^T.$$
Then $I_3(t_n)$ reads
\bes
\begin{split}
    I_3(t_n) & = \int_0^\tau \fe^{\rho \alpha\p_x} g \beta \fe^{-\rho\alpha\p_x} \Phi(t_n) \,d\rho\\
    & = \int_0^\tau \left[\fe^{\rho\p_x}g\fe^{\rho\p_x}\beta\Pi_- + \fe^{-\rho\p_x}g\fe^{-\rho\p_x}\beta\Pi_+\right]\Phi(t_n)\,d\rho,
\end{split}
\ees
which as we shall see can also be done explicitly in the physical space. Let us consider three general complex-valued scalar functions $ \psi:=\psi(x),~\phi:=\phi(x),~\varphi:=\varphi(x)$, and we find
\begin{align*}
 & \int_0^\tau\fe^{\rho\partial_x}[
(\fe^{\rho\partial_x}\psi)
(\fe^{-\rho\partial_x}\phi)\fe^{\rho\partial_x}\varphi]\, d\rho
= \int_0^\tau\sum_{l\in\bZ}\fe^{ilx}\fe^{il\rho}
 \sum_{\begin{array}{cc}&l_1,l_2,l_3\in\bZ\\&l_1+l_2+l_3=l\end{array}}
\fe^{il_1\rho}\widehat{\psi}_{l_1}\fe^{-il_2\rho}\widehat{\phi}_{l_2}
\fe^{il_3\rho}\widehat{\varphi}_{l_3}\, d\rho\\
&= \sum_{l\in\bZ}
 \sum_{\begin{array}{cc}&l_1,l_2,l_3\in\bZ\\&l_1+l_2+l_3=l\end{array}}
\fe^{ilx}\int_0^\tau
\fe^{2i(l_1+l_3)\rho}d\rho\widehat{\psi}_{l_1}\widehat{\phi}_{l_2}\widehat{\varphi}_{l_3}
=\tau\phi\varphi_1(2\tau\partial_x)(\psi\varphi).
\end{align*}
Similarly,
$$\int_0^\tau\fe^{-\rho\partial_x}[
(\fe^{\rho\partial_x}\psi)
(\fe^{-\rho\partial_x}\phi)\fe^{-\rho\partial_x}\varphi]\, d\rho
=\tau\psi\varphi_1(-2\tau\partial_x)(\phi\varphi).$$
Hence, we have for $I_3(t_n)$:
\be\label{eq:I3}
\begin{split}
    I_3(t_n) & = \frac{\tau \lambda}{2}\Big[ \overline{\phi_+}(t_n)\varphi_1(2\tau\p_x)\left(\phi_-(t_n) \beta\Pi_-\Phi(t_n)\right) + \phi_+(t_n)\varphi_1(2\tau\p_x)\left(\overline{\phi_-}(t_n) \beta\Pi_-\Phi(t_n)\right) \\
    & \quad + \phi_-(t_n)\varphi_1(-2\tau\p_x)\left(\overline{\phi_+}(t_n) \beta\Pi_+\Phi(t_n)\right) + \overline{\phi_-}(t_n)\varphi_1(-2\tau\p_x)\left(\phi_+(t_n) \beta\Pi_+\Phi(t_n)\right)\Big].
\end{split}
\ee

In summary of (\ref{sec3 eq1})-(\ref{eq:I3}), the detailed \textbf{scheme of the first-order ultra low-regularity integrator} \textbf{(ULI) for integrating NDE} (\ref{dirac trun}) \textbf{with given $V$} reads: denote  $\Phi^n=(\phi_1^n, \phi_2^n)^T\approx\Phi(t_n)$ for $n\geq0$,
let $\Phi^0=\Phi_0$ and then
\begin{align} \label{ULI}
\Phi^{n+1}=&\fe^{-\tau\alpha\partial_x}\Phi^n-i\fe^{-\tau\alpha\partial_x}(I_1^n+I_2^n+I_3^n)
=:\Theta_{\mathrm{ext}}(\Phi^{n}),
\quad n\geq0,
 \end{align}
 where
 \begin{align}
 &I_1^n=\tau\left[ \varphi_1(2\tau\p_x)\beta\Pi_- + \varphi_2(-2\tau\p_x)\beta\Pi_+\right] \Phi^n,\quad
 I_2^n= \tau \left[(\varphi_1(\tau\p_x) V) \Pi_+ + (\varphi_1(-\tau\p_x) V)\Pi_-\right]\Phi^n,\nonumber\\
&I_3^n=\frac{\tau\lambda}{4}
\begin{pmatrix}\overline{\phi_+^n}\varphi_1(2\tau\partial_x)(\phi_-^n)^2
+\phi_+^n\varphi_1(2\tau\partial_x)|\phi_-^n|^2
+\phi_-^n\varphi_1(-2\tau\partial_x)|\phi_+^n|^2
+\overline{\phi_-^n}\varphi_1(-2\tau\partial_x)(\phi_+^n)^2\\
\overline{\phi_+^n}\varphi_1(2\tau\partial_x)(\phi_-^n)^2
+\phi_+^n\varphi_1(2\tau\partial_x)|\phi_-^n|^2
-\phi_-^n\varphi_1(-2\tau\partial_x)|\phi_+^n|^2
-\overline{\phi_-^n}\varphi_1(-2\tau\partial_x)(\phi_+^n)^2
\end{pmatrix},\label{ULI-I13}
 \end{align}
 with
 $\phi^n_\pm=\phi_1^n\pm\phi_2^n$.

The proposed ULI scheme, i.e., \eqref{ULI} with \eqref{ULI-I13}, is fully explicit. In practical computations, the spatial discretization of ULI could easily be done by the Fourier pseudospectral method \cite{ST}, where the computational cost at each time level is $O(N\log{N})$ with $N$ the number of the total Fourier modes. Thus, the ULI scheme is of similar computational costs as the standard methods in Section~\ref{sec:2}.

\subsection{NDE with consistent field} Next, we present the integration for the NDE coupled with the Poisson equation:
\begin{subequations}\label{dirac-poisson}
\begin{align}
&i\partial_{t}\Phi=-i\alpha\partial_x\Phi+\beta\Phi +V\Phi+F(\Phi),\quad t>0,\ x\in\bT,\label{dirac-poisson trun a}\\
&-\partial_{xx} V=|\Phi|^2,\quad t\geq0,\ x\in\bT,\quad \int_{\bT}Vdx=0,\quad t\geq0,\\
&\Phi(0)=\Phi_0,\quad x\in\bT,
\end{align}
\end{subequations}
where periodic boundary conditions are imposed for both $\Phi$ and $V$.

The Duhamel's formula for (\ref{dirac-poisson}) reads as
\begin{align}\label{dumammel-DP}
  \Phi(t_n+s)=\fe^{-s\alpha\partial_x}\Phi(t_n)-i\int_0^s\fe^{-(s-\rho)\alpha\partial_x}
 \left[\beta\Phi(t_n+\rho) +V(t_n+\rho)\Phi(t_n+\rho) + F(\Phi(t_n+\rho))\right] d\rho.
\end{align}
Let $s=\tau$ in  (\ref{dumammel-DP}), and by adopting $\Phi(t_n+\rho)\approx\fe^{-\rho\alpha\partial_x}\Phi(t_n)$ as before,  the
approximation goes the same as (\ref{sec3 eq1}) but with $I_2(t_n)$ (re-)defined as:
\begin{equation}\label{I2dp0}
I_2(t_n):=-\int_0^\tau\fe^{\rho\alpha\partial_x}
\left(\partial_{xx}^{-1}|\fe^{-\rho \alpha \p_x}\Phi(t_n)|^2\right)\fe^{-\rho\alpha\partial_x}\Phi(t_n)\,  d\rho.\end{equation}
Here, we define for some general function $\psi(x):\bT\to\bC$,
\begin{equation*}
 \partial_{xx}^{-1}\psi(x):=-\sum_{l\neq0}\frac{1}{l^2}\widehat{\psi}_l\fe^{ilx},
\end{equation*}
where $\partial_{xx}^{-1}$ denotes the {\em natural}  inverse operator of $\partial_{xx}$.

The integrations of $I_1(t_n)$ and $I_3(t_n)$ remain the same. For $I_2(t_n)$, firstly we have
\bes
\begin{split}
    I_2(t_n)
    & = -\int_0^\tau \left[\fe^{\rho\p_x}\left(\partial_{xx}^{-1}|\fe^{-\rho \alpha \p_x}\Phi(t_n)|^2\right)\fe^{-\rho\p_x}\Pi_+ + \fe^{-\rho\p_x}\left(\partial_{xx}^{-1}|\fe^{-\rho \alpha \p_x}\Phi(t_n)|^2\right)\fe^{\rho\p_x}\Pi_-\right]\Phi(t_n)\,d\rho,
\end{split}
\ees
and
$$|\fe^{-\rho \alpha \p_x}\Phi(t_n)|^2=
\frac{1}{2}\left[(\fe^{\rho\partial_x}\phi_-(t_n))(\fe^{\rho\partial_x}\overline{
\phi_-}(t_n))+(\fe^{-\rho\partial_x}\phi_+(t_n))(\fe^{-\rho\partial_x}\overline{
\phi_+}(t_n))\right].$$
Let us consider two general scalar functions $\psi=\psi(x),\,\phi=\phi(x)$, and we have
\begin{align*}
 & \int_0^\tau\fe^{\rho\partial_x}\left(\partial_{xx}^{-1}[
(\fe^{\rho\partial_x}\psi)
(\fe^{\rho\partial_x}\overline{\psi})]\right)\fe^{-\rho\partial_x}\phi\, d\rho\\
=& \int_0^\tau\sum_{l\in\bZ}\fe^{ilx}\fe^{il\rho}
 \sum_{\begin{array}{cc}&l_1,l_2,l_3\in\bZ\\&l_1+l_2+l_3=l\\&l_1+l_2\neq0
 \end{array}}\frac{-1}{(l_1+l_2)^2}
\fe^{il_1\rho}\widehat{\psi}_{l_1}\fe^{il_2\rho}\widehat{(\overline{\psi})}_{l_2}
\fe^{-il_3\rho}\widehat{\phi}_{l_3}\, d\rho\\
=& \sum_{l\in\bZ}\fe^{ilx}
 \sum_{\begin{array}{cc}&l_1,l_2,l_3\in\bZ\\&l_1+l_2+l_3=l\\&l_1+l_2\neq0\end{array}}
\int_0^\tau
\frac{-\fe^{2i(l_1+l_2)\rho}}{(l_1+l_2)^2}d\rho\,
\widehat{\psi}_{l_1}\widehat{(\overline{\psi})}_{l_2}\widehat{\phi}_{l_3}
=\tau\phi\varphi_1(2\tau\partial_x)\partial_{xx}^{-1}|\psi|^2,
\end{align*}
and
\begin{align*}
 & \int_0^\tau\fe^{\rho\partial_x}\left(\partial_{xx}^{-1}[
(\fe^{-\rho\partial_x}\psi)
(\fe^{-\rho\partial_x}\overline{\psi})]\right)\fe^{-\rho\partial_x}\phi\, d\rho\\
=& \int_0^\tau\sum_{l\in\bZ}\fe^{ilx}\fe^{il\rho}
 \sum_{\begin{array}{cc}&l_1,l_2,l_3\in\bZ\\&l_1+l_2+l_3=l\\&l_1+l_2\neq0
 \end{array}}\frac{-1}{(l_1+l_2)^2}
\fe^{-il_1\rho}\widehat{\psi}_{l_1}\fe^{-il_2\rho}\widehat{(\overline{\psi})}_{l_2}
\fe^{-il_3\rho}\widehat{\phi}_{l_3}\, d\rho\\
= &\sum_{l\in\bZ}\fe^{ilx}
 \sum_{\begin{array}{cc}&l_1,l_2,l_3\in\bZ\\&l_1+l_2+l_3=l\\&l_1+l_2\neq0\end{array}}
\frac{-\tau}{(l_1+l_2)^2}
\widehat{\psi}_{l_1}\widehat{(\overline{\psi})}_{l_2}\widehat{\phi}_{l_3}
=\tau\phi\partial_{xx}^{-1}|\psi|^2.
\end{align*}
Similarly,
\begin{align*}
 & \int_0^\tau\fe^{-\rho\partial_x}\left(\partial_{xx}^{-1}[
(\fe^{\rho\partial_x}\psi)
(\fe^{\rho\partial_x}\overline{\psi})]\right)\fe^{\rho\partial_x}\phi\, d\rho
=\tau\phi\partial_{xx}^{-1}|\psi|^2,\\
&\int_0^\tau\fe^{-\rho\partial_x}\left(\partial_{xx}^{-1}[
(\fe^{-\rho\partial_x}\psi)
(\fe^{-\rho\partial_x}\overline{\psi})]\right)\fe^{\rho\partial_x}\phi\, d\rho
=\tau\phi\varphi_1(-2\tau\partial_x)\partial_{xx}^{-1}|\psi|^2.
\end{align*}
Hence, we have for $I_2(t_n)$:
\be\label{eq:I2dp}
\begin{split}
    I_2(t_n) = & -\frac{\tau}{2}\Pi_+\Phi(t_n)\partial_{xx}^{-1}\left[
    \varphi_1(2\tau\p_x)|\phi_-(t_n)|^2 + |\phi_+(t_n)|^2 \right] \\
    &  -\frac{\tau}{2}\Pi_-\Phi(t_n)\partial_{xx}^{-1}\left[
    \varphi_1(-2\tau\p_x)|\phi_+(t_n)|^2 + |\phi_-(t_n)|^2 \right].
\end{split}
\ee

Combining (\ref{sec3 eq1}), (\ref{eq:I1}), (\ref{eq:I3}) and (\ref{eq:I2dp}), the detailed \textbf{scheme of ULI for solving the Dirac-Poisson} system (\ref{dirac-poisson}) reads  the same as (\ref{ULI}):
\begin{align} \label{ULIdp}
\Phi^{n+1}=&\fe^{-\tau\alpha\partial_x}\Phi^n-i\fe^{-\tau\alpha\partial_x}(I_1^n+I_2^n+I_3^n)
=:\Theta_{\mathrm{DP}}(\Phi^{n}),
\quad n\geq0,
 \end{align}
where $I_1^n,I_3^n$ are defined same as in (\ref{ULI-I13}), but $I_2^n$ is replaced by
 \begin{align}\label{I2dp}
 I_2^n= -\frac{\tau}{4}
    \begin{pmatrix}\phi_+^n\partial_{xx}^{-1}\left[
    \varphi_1(2\tau\p_x)|\phi_-^n|^2 + |\phi_+^n|^2 \right] +\phi^n_-\partial_{xx}^{-1}\left[
    \varphi_1(-2\tau\p_x)|\phi_+^n|^2 + |\phi_-^n|^2 \right]\\
    \phi_+^n\partial_{xx}^{-1}\left[
    \varphi_1(2\tau\p_x)|\phi_-^n|^2 + |\phi_+^n|^2 \right] -\phi^n_-\partial_{xx}^{-1}\left[
    \varphi_1(-2\tau\p_x)|\phi_+^n|^2 + |\phi_-^n|^2 \right]\end{pmatrix}
    .
 \end{align}

 The above ULI scheme, i.e., (\ref{ULIdp}) with (\ref{I2dp}), is again fully explicit with computational costs of $O(N\log N)$ at each time level if $N$ Fourier modes are used for spatial discretization.

\subsection{Convergence result}
For the proposed ULI scheme, i.e., (\ref{ULI}) with (\ref{ULI-I13}) for solving the NDE (\ref{dirac trun}) with external $V$ or (\ref{ULIdp}) with (\ref{I2dp}) for solving the Dirac-Poisson system (\ref{dirac-poisson}), we have the following convergence result  as the\textbf{ main theorem} of the paper.

\begin{theorem}\label{main}(Convergence of ULI) Let $\Phi^n$ denote the numerical solution of the ULI scheme (\ref{ULIdp}) for solving the Dirac-Poisson system (\ref{dirac-poisson}) (respectively, of  the ULI scheme (\ref{ULI})  for solving (\ref{dirac trun}) with given $V$). Let $r>\frac{1}{2}$ and $\Phi\in L^\infty((0,T);H^{r})$ for some $T>0$. Then there exist constants $\tau_0,\,C>0$ depending on $ \Vert \Phi\Vert_{L^\infty((0,T);H^r)}$ and $T$, such that for all $0<\tau\leq\tau_0$ and $0\leq t_n\leq T$, we have
 $$\|\Phi(t_n)-\Phi^n\|_{r}\leq C\tau.$$
\end{theorem}

\begin{proof} We shall prove the case for Dirac-Poisson system, and the proof for the other case is similar which is omitted here for brevity.

\textbf{Local error}. Define $\xi^n:=\Phi(t_{n+1})-\Theta_{\mathrm{DP}}(\Phi(t_n))$ as the local truncation error of the scheme at some $t_n$ for $n\geq0$. Let $s=\tau$ in (\ref{dumammel-DP}) and subtract it from $\Theta_{\mathrm{DP}}(\Phi(t_n))$, noting the definition of $I_1(t_n),\,I_2(t_n),\,I_3(t_n)$ in (\ref{I13}) and (\ref{I2dp0}) which are exactly evaluated in the scheme $\Theta_{\mathrm{DP}}$, we get
  \begin{align*}
    \xi^n=&-i\int_0^\tau\fe^{-(s-\rho)\alpha\partial_x}
 \Big[\beta\Phi(t_n+\rho)-\beta\fe^{-\rho\alpha\partial_x}\Phi(t_n)-\partial_{xx}^{-1}
 |\Phi(t_n+\rho)|^2\Phi(t_n+\rho)\\
  &-(-\partial_{xx}^{-1})|\fe^{-\rho\alpha\partial_x}\Phi(t_n)|^2
  \fe^{-\rho\alpha\partial_x}\Phi(t_n)  + F(\Phi(t_n+\rho))
  -F(\fe^{-\rho\alpha\partial_x}\Phi(t_n))\Big] d\rho.
  \end{align*}
  Based on (\ref{eq:3-3}), it is direct to verify that
  $\fe^{s\alpha\partial_x}$ is isometric in $H^r$ space, i.e.,
  $$|\widehat{\left(\fe^{s\alpha\partial_x}\Psi\right)}_l|^2=|\widehat{\Psi}_l|^2,
  \quad l\in\bZ,\ s\in\bR,\quad \mbox{where}\quad
  \widehat{\Psi}_l=\int_\bT \Psi(x)\fe^{-ilx}dx,$$
  for a general function $\Psi(x):\bT\to\bC^2$.
  Then thanks to the triangle inequality and bilinear estimates since $r>\frac{1}{2}$, we have
  \begin{align*}
    \|\xi^n\|_{r}\leq&\int_0^\tau
 \|\Phi(t_n+\rho)-\fe^{-\rho\alpha\partial_x}\Phi(t_n)\|_{r}d\rho\\
 &+ \int_0^\tau\|\partial_{xx}^{-1}
 (|\Phi(t_n+\rho)|^2)\Phi(t_n+\rho)
  -\partial_{xx}^{-1}(|\fe^{-\rho\alpha\partial_x}\Phi(t_n)|^2)
  \fe^{-\rho\alpha\partial_x}\Phi(t_n)\|_{r}d\rho\\
  &  + \int_0^\tau\|F(\Phi(t_n+\rho))
  -F(\fe^{-\rho\alpha\partial_x}\Phi(t_n))\|_{r} d\rho\\
  \leq& C\int_0^\tau
 \|\Phi(t_n+\rho)-\fe^{-\rho\alpha\partial_x}\Phi(t_n)\|_{r}d\rho.\end{align*}
 By (\ref{dumammel-DP}) again, we see that
 \begin{align*}
  &\|\Phi(t_n+\rho)-\fe^{-\rho\alpha\partial_x}\Phi(t_n)\|_{r}\\
  \leq& \int_0^\rho\left[\|\Phi(t_n+\sigma)\|_{r}
  + \|\partial_{xx}^{-1}
 (|\Phi(t_n+\sigma)|^2)\Phi(t_n+\sigma)\|_{r}
  +\|F(\Phi(t_n+\sigma))\|_{r}\right]d\sigma\\
  \leq&C\int_0^\rho\|\Phi(t_n+\sigma)\|_{r}d\sigma\leq C\rho\|\Phi\|_{L^\infty((0,T);H^r)},
 \end{align*}
and therefore
 \begin{align}\label{locaerror1}
\|\xi^n\|_{r}\leq&C\int_0^\tau\rho d\rho\|\Phi\|_{L^\infty((0,T);H^r)}\leq C\tau^2.
  \end{align}

We carry out an induction proof on the boundedness of the numerical solution (see also the \emph{Lady Windermere's fan} argument e.g. \cite{Faou,Lubich,lownls}). Assume that for some $0\leq m<T/\tau$,
$$\|\Phi^n\|_r\leq \|\Phi\|_{L^\infty((0,T);H^r)}+1,\quad 0\leq n\leq m,$$
which is obviously true for $m=0$. Now we justify it for $n=m+1$.

  \textbf{Stability \& convergence}. Taking the difference between the scheme (\ref{ULIdp}) and $\Phi(t_{n+1})=\Theta_{DP}(\Phi(t_n))+\xi^n$ and denoting $e^n=\Phi(t_n)-\Phi^n$ for $n\geq0$, we get
  \begin{align*}
    e^{n+1}=\fe^{-\tau\alpha\partial_x}e^n-i\fe^{-\tau\alpha\partial_x}
    \left[I_1(t_n)-I_1^n+I_2(t_n)-I_2^n+I_3(t_n)-I_3^n\right]+\xi^n,\quad 0\leq n\leq m.
  \end{align*}
  Taking the $H^r$-norm on both sides of the above equation and by triangle inequality,
  we get
  \begin{align}\label{proof1}
    \|e^{n+1}\|_{r}\leq\|e^n\|_r+
    \|I_1(t_n)-I_1^n\|_r+\|I_2(t_n)-I_2^n\|_r+
    \|I_3(t_n)-I_3^n\|_r+\|\xi^n\|_r,\quad 0\leq n\leq m.
  \end{align}
  Noting in $I_1^n,I_2^n,I_3^n$, we have
  $$\|\varphi_1(s\partial_x)\psi\|_r\leq C\|\psi\|_r,\quad s\in\bR,$$
 for some $\psi\in H^r(\bT)$. Then by direct computing and the bilinear estimates, and thanks to the fact that $\Phi^n\in H^r$ for $n\leq m$, we have
 $$\|I_1(t_n)-I_1^n\|_r+\|I_2(t_n)-I_2^n\|_r+
    \|I_3(t_n)-I_3^n\|_r\leq \tau C\|e^n\|_r,\quad 0\leq n\leq m.$$
    Therefore, by plugging the above inequality and the local truncation error (\ref{locaerror1}) into (\ref{proof1}), we get
 \begin{align*}
    \|e^{n+1}\|_{r}\leq\|e^n\|_r+\tau C
    \|e^n\|_r+C\tau^2,\quad 0\leq n\leq m.
  \end{align*}
Then by Gronwall's inequality, we have
\begin{align*}
    \|e^{m+1}\|_{r}\leq C\tau,
  \end{align*}
for some constant $C>0$ dependent on $T$ and norm of $\Phi$ but independent of $m$ or $\tau$. Then there exists some constant $\tau_0>0$ independent of $m$ or $\tau$, such that
  $$\|\Phi^{m+1}\|_r\leq\|e^{m+1}\|_{r}+ \|\Phi\|_{L^\infty((0,T);H^r)}\leq
  \|\Phi\|_{L^\infty((0,T);H^r)}+1,$$
  and induction proof is done.
\end{proof}
\begin{remark}
  We remark that the assumption $r>\frac{1}{2}$ in Theorem \ref{main} is necessary to apply classical bilinear estimates in the error analysis, i.e., we can exploit that $H^r$ is then an algebra. This stability restriction (in principle) can be weakened by using discrete Strichartz-type estimates, see e.g. the recent work on the nonlinear Schr\"{o}dinger equation \cite{lowNLS2}.
\end{remark}

\begin{remark}
  The generalizations of the ULI to higher dimensional Dirac equations are not straightforward. In 2D or 3D, the propagator in the Duhamel's
formula contains spatial differential operators in a nonlinear way, where more efforts are needed to integrate the nonlinearity after the Picard iteration. This difficulty also appears in the counterpart study for multidimensional wave equations and we are going to address it in a forthcoming paper.
\end{remark}

\section{Extension to higher order}\label{sec:4}
In this section, we present second-order ULI schemes for solving the NDEs. Again, we begin with the external electrical field case (\ref{dirac trun}) and then consider the Dirac-Poisson system (\ref{dirac-poisson}).

\subsection{NDE with external field} Assume $V=V_e(x)$ is given.
In principle, the ULI scheme could be extended to arbitrary high order by using  the nested Picard iteration, i.e., use recursively the lower-order scheme for approximating $\Phi(t_n+\rho)$ in the integrand of the Duhamel's formula:
\begin{align}\label{dumammel2}
\Phi(t_{n+1})=\fe^{-\tau\alpha\partial_x}\Phi(t_n)-i\int_0^\tau\fe^{-(\tau-\rho)\alpha\partial_x}
 \left[\mathcal{L}\Phi(t_n+\rho) +\mathcal{N}(\Phi(t_n+\rho))\right] d\rho,
 \end{align}
 where in this case
 $$\mathcal{L}=\beta+V,\quad \mathcal{N}(\Phi)=\lambda (\Phi^*\beta\Phi)\beta\Phi=F(\Phi).$$
 For instance, to get a second-order ULI scheme, one can take
$$\Phi(t_n+\rho)\approx\Theta_{\mathrm{ext}}(\Phi(t_n)),$$
with the mapping $\Theta_{\mathrm{ext}}$ from the first-order scheme (\ref{ULI}), and then carry out the integrations in Fourier frequency space exactly. However, it is not clear how to cope with the  produced pseudo differential operators, especially, in regard of the practical implementation of the scheme. To develop  a  low-regularity second-order scheme which is of  comparable costs to classical methods, i.e., of order $O(N\log{N})$ with $N$ the number of the total Fourier modes, we accept to introduce some truncations that involve first-order spatial derivatives of the solution:
\begin{equation}\label{eqadd1 sec4}\Phi(t_n+\rho)\approx\fe^{-\rho\alpha\partial_x}\Phi(t_n)
-i\rho G(t_n),\quad G(t_n)=\mathcal{L}\Phi(t_n)+\mathcal{N}(\Phi(t_n)).\end{equation}
 The loss of one derivative seems to be acceptable for second-order time convergence, in particular  since we also need some smoothness of the solution for  spatial discretization accuracy. It is not practically meaningful from the efficiency point of view if the convergence order of the method in space is much less than the temporal convergence order.

Plugging the approximation (\ref{eqadd1 sec4}) into (\ref{dumammel2}), we approximate the first part as
\begin{align*}&-i\int_0^\tau\fe^{-(\tau-\rho)\alpha\partial_x}
\mathcal{L}\Phi(t_n+\rho) d\rho\\
\approx&-i\int_0^\tau\fe^{-(\tau-\rho)\alpha\partial_x}
\mathcal{L}\fe^{-\rho\alpha\partial_x}\Phi(t_n)d\rho
-\int_0^\tau\fe^{-(\tau-\rho)\alpha\partial_x}\rho\mathcal{L}
G(t_n) d\rho\\
\approx&-i\int_0^\tau\fe^{-(\tau-\rho)\alpha\partial_x}
\mathcal{L}\fe^{-\rho\alpha\partial_x}\Phi(t_n)d\rho
-\frac{\tau^2}{2}\mathcal{L}
G(t_n)\\
=&- i \fe^{-\tau\alpha\p_x} [I_1(t_n)+I_2(t_n)]-\frac{\tau^2}{2}\mathcal{L}
G(t_n),
 \end{align*}
 where $I_1$ and $I_2$ are defined in (\ref{I13}).
 For the second part, we take the approximation as
 \begin{align*}
 \mathcal{N}(\Phi(t_n+\rho) )&\approx
 \mathcal{N}\left(\fe^{-\rho\alpha\partial_x}\Phi(t_n)
-i\rho G(t_n)\right)\\
&=
 \mathcal{N}(\fe^{-\rho\alpha\partial_x}\Phi(t_n))
-\rho\mathcal{N}'(\fe^{-\rho\alpha\partial_x}\Phi(t_n))(i G(t_n))+O(\rho^2)\\
&=
 \mathcal{N}(\fe^{-\rho\alpha\partial_x}\Phi(t_n))
-\rho\mathcal{N}'(\Phi(t_n))(i G(t_n))+O(\rho^2),
 \end{align*}
 where $\mathcal{N}'$ denotes the G\^{a}teaux's derivative, and we have for $G=(G_1,G_2)^T$,
\begin{align*}
\mathcal{N}'(\Phi(t_n))(iG(t_n))=&i\lambda(|\phi_1(t_n)|^2-|\phi_2(t_n)|^2)\beta G(t_n)\\
&+
2\lambda\mathrm{Im}
\left[\phi_1(t_n)\overline{G_1}(t_n)+\overline{\phi_2}(t_n)G_2(t_n)\right]\beta\Phi(t_n).
\end{align*}
 Therefore, we have
 \begin{align*}&-i\int_0^\tau\fe^{-(\tau-\rho)\alpha\partial_x}
\mathcal{N}(\Phi(t_n+\rho) )d\rho\\
\approx&-i\int_0^\tau\fe^{-(\tau-\rho)\alpha\partial_x}
\mathcal{N}(\fe^{-\rho\alpha\partial_x}\Phi(t_n))d\rho
+i\int_0^\tau\fe^{-(\tau-\rho)\alpha\partial_x}\rho
d\rho\,\mathcal{N}'(\Phi(t_n))(iG(t_n))\\
\approx&-i\int_0^\tau\fe^{-(\tau-\rho)\alpha\partial_x}
\mathcal{N}(\fe^{-\rho\alpha\partial_x}\Phi(t_n))d\rho
+\frac{i\tau^2}{2}\mathcal{N}'(\Phi(t_n))(iG(t_n))\\
=&
-\frac{\lambda\tau^2}{2} (|\phi_1(t_n)|^2-|\phi_2(t_n)|^2)\beta G(t_n)+i\lambda\tau^2
\mathrm{Im}
\left[\phi_1(t_n)\overline{G_1}(t_n)+\overline{\phi_2}(t_n)G_2(t_n)\right]\beta\Phi(t_n)\\
&-i\fe^{-\tau\alpha\partial_x}
I_3(t_n),
\end{align*}
where $I_3$ is defined in (\ref{I13}).
Then the complete second-order approximation to (\ref{dumammel2}) is
\begin{align*}
\Phi(t_{n+1})\approx&
\fe^{-\tau\alpha\partial_x}\Phi(t_n)
-i\fe^{-\tau\alpha\partial_x}\left[I_1(t_n)+I_2(t_n)+I_3(t_n)\right]
-\frac{\tau^2}{2}\mathcal{L}
G(t_n)\\
&-\frac{\lambda\tau^2}{2} (|\phi_1(t_n)|^2-|\phi_2(t_n)|^2)\beta G(t_n)
+i\lambda\tau^2
\mathrm{Im}
\left[\phi_1(t_n)\overline{G_1}(t_n)+\overline{\phi_2}(t_n)G_2(t_n)\right]\beta\Phi(t_n).
 \end{align*}

The detailed \textbf{scheme for the second-order ULI for NDE (\ref{dirac trun})} \textbf{with given $V$} reads: denote  $\Phi^n=(\phi_1^n, \phi_2^n)^T\approx\Phi(t_n)$ for $n\geq0$, let $\Phi^0=\Phi_0$ and then
\begin{align}\label{LI}
\Phi^{n+1}=&
\Theta_{\mathrm{ext}}(\Phi^n)
-\frac{\tau^2}{2}\mathcal{L}G^n
-\frac{\lambda\tau^2}{2} (|\phi_1^n|^2-|\phi_2^n|^2)\beta G^n
+i\lambda\tau^2
\mathrm{Im}
[\phi_1^n\overline{G_1^n}+\overline{\phi_2^n}G_2^n]\beta\Phi^n,\ n\geq0,
 \end{align}
 where
 $$G^n=(G_1^n,G_2^n)^T=\beta\Phi^n+V\Phi^n+F(\Phi^n).$$

\subsection{Dirac-Poisson system} For the Dirac-Poisson system (\ref{dirac-poisson}), the linear and nonlinear operators $\mathcal{L}$ and $\mathcal{N}$ in the Duhamel's formula (\ref{dumammel2}) are redefined  as
$$\mathcal{L}=\beta,\quad \mathcal{N}(\Phi)=-\left(\partial_{xx}^{-1}|\Phi|^2\right)\Phi+
\lambda(|\phi_1|^2-|\phi_2|^2)\beta\Phi,$$
where now
\begin{align*}
\mathcal{N}'(\Phi(t_n))(iG(t_n))=&-\left(\partial_{xx}^{-1}|\Phi(t_n)|^2\right)(iG(t_n))+
i\lambda(|\phi_1(t_n)|^2-|\phi_2(t_n)|^2)\beta G(t_n)\\
&-2\left(\partial_{xx}^{-1}\mathrm{Im}
[\phi_1(t_n)\overline{G_1}(t_n)-\overline{\phi_2}(t_n)G_2(t_n)]\right)\Phi(t_n)\\
&+
2\lambda\mathrm{Im}
[\phi_1(t_n)\overline{G_1}(t_n)+\overline{\phi_2}(t_n)G_2(t_n)]\beta\Phi(t_n).
\end{align*}
Then we have
\begin{align*}&-i\int_0^\tau\fe^{-(\tau-\rho)\alpha\partial_x}
\mathcal{L}\Phi(t_n+\rho) d\rho
\approx- i \fe^{-\tau\alpha\p_x} I_1(t_n)-\frac{\tau^2}{2}\beta
G(t_n),
 \end{align*}
 and
  \begin{align*}&-i\int_0^\tau\fe^{-(\tau-\rho)\alpha\partial_x}
\mathcal{N}(\Phi(t_n+\rho) )d\rho\\
\approx&
-\frac{\lambda\tau^2}{2} (|\phi_1(t_n)|^2-|\phi_2(t_n)|^2)\beta G(t_n)+i\lambda\tau^2
\mathrm{Im}
\left[\phi_1(t_n)\overline{G_1}(t_n)+\overline{\phi_2}(t_n)G_2(t_n)\right]\beta\Phi(t_n)\\
&+\frac{\tau^2}{2}\left(\partial_{xx}^{-1}|\Phi(t_n)|^2\right)G(t_n)
-i\tau^2\left(\partial_{xx}^{-1}\mathrm{Im}
[\phi_1(t_n)\overline{G_1}(t_n)-\overline{\phi_2}(t_n)G_2(t_n)]\right)\Phi(t_n)\\
&-i\fe^{-\tau\alpha\partial_x}\left[I_2(t_n)+I_3(t_n)\right],
\end{align*}
 where $I_1,\,I_3$ are defined in (\ref{I13}) and $I_2$ is defined in (\ref{I2dp0}).
Then the detailed \textbf{scheme for the second-order ULI for Dirac-Poisson system (\ref{dirac-poisson})} reads: let $\Phi^0=\Phi_0$ and then
\begin{align}\label{LIdp}
\Phi^{n+1}=&
\Theta_{\mathrm{DP}}(\Phi^n)
-\frac{\tau^2}{2}\beta G^n
-\frac{\lambda\tau^2}{2} (|\phi_1^n|^2-|\phi_2^n|^2)\beta G^n
+i\lambda\tau^2
\mathrm{Im}
[\phi_1^n\overline{G_1^n}+\overline{\phi_2^n}G_2^n]\beta\Phi^n\nonumber\\
&+\frac{\tau^2}{2}\left(\partial_{xx}^{-1}|\Phi^n|^2\right)G^n
-i\tau^2\left(\partial_{xx}^{-1}\mathrm{Im}
[\phi_1^n\overline{G_1^n}-\overline{\phi_2^n}G_2^n]\right)\Phi^n
=:\Theta_{\mathrm{DP}}^{2nd}(\Phi^n),\quad n\geq0,
 \end{align}
 where
 $$G^n=(G_1^n,G_2^n)^T=\beta\Phi^n-(\partial_{xx}^{-1}|\Phi^n|^2)\Phi^n+F(\Phi^n).$$

The extended ULI schemes (\ref{LI}) and (\ref{LIdp}) for solving respectively (\ref{dirac trun}) and (\ref{dirac-poisson}) are fully explicit and easy to program since they are built based on the first-order ULI schemes. The computational cost of (\ref{LI}) or (\ref{LIdp}) per time level is also $O(N\log N)$ if $N$ Fourier modes are used for spatial discrezation which is as efficient as the standard methods.

\subsection{Convergence result} For the extended ULI schemes (\ref{LI}) and (\ref{LIdp}), we have the following convergence theorem.

 \begin{theorem}\label{main2} (Convergence of extended ULI)
 Let $\Phi^n$ denote the numerical solution  of the second-order ULI scheme (\ref{LIdp}) for solving the Dirac-Poisson system (\ref{dirac-poisson}) (respectively, of (\ref{LI})  for solving (\ref{dirac trun}) with given $V$).  Let $r>\frac{1}{2}$, $\Phi\in L^\infty((0,T);H^{r+1})$ and $\partial_t\Phi\in L^\infty((0,T);H^{r})$ for some $T>0$. Then there exist constants $\tau_0,\,C>0$ depending on $ \Vert \Phi\Vert_{L^\infty((0,T);H^{r+1})}$, $ \Vert \partial_t\Phi\Vert_{L^\infty((0,T);H^{r})}$ and $T$, such that for all $0<\tau\leq\tau_0$ and $0\leq t_n\leq T$, we have
 $$\|\Phi(t_n)-\Phi^n\|_{r}\leq C\tau^2.$$
\end{theorem}

\begin{proof}
\textbf{Local error}. Define $\xi^n=\Phi(t_{n+1})-\Theta_{\mathrm{DP}}^{2nd}(\Phi(t_n))$ as the local truncation error of the scheme (\ref{LIdp}) at some $t_n$ for $n\geq0$.
Firstly, we denote
\begin{equation}\label{proof2eq1}\xi_1^n(\rho)=\Phi(t_n+\rho)-\fe^{-\rho\alpha\partial_x}\Phi(t_n)
+i\rho G(t_n),\end{equation}
which by the Duhamel's formula and Taylor expansion becomes
\begin{align*}\xi_1^n(\rho)=&
-i\int_0^\rho\fe^{-(\rho-\sigma)\alpha\partial_x}
 G(t_n+\sigma) d\sigma+i\rho G(t_n)\\
=&i\rho G(t_n)-i\int_0^\rho G(t_n+\sigma)d\sigma-i\int_0^\rho\rho\int_1^{\sigma/\rho}
\fe^{-\rho(1-\kappa)\alpha\partial_x}\alpha\partial_x G(t_n+\sigma)d\kappa d\sigma\\
=&-i\int_0^\rho\sigma\int_0^1\partial_tG(t_n+\sigma \kappa)d\kappa d\sigma-i\int_0^\rho\rho\int_1^{\sigma/\rho}
\fe^{-\rho(1-\kappa)\alpha\partial_x}\alpha\partial_x G(t_n+\sigma)d\kappa d\sigma,
 \end{align*}
 and so we have
 $$\|\xi_1^n(\rho)\|_r\leq C\tau^2\left(\|\partial_t\Phi\|_{L^\infty((0,T),H^r)}
 +\|\Phi\|_{L^\infty((0,T),H^{r+1})}\right)\leq C\tau^2,\quad 0\leq \rho\leq\tau. $$
Also by Taylor expansion,
 \begin{align}
 \mathcal{N}\left(\Phi(t_n+\rho)\right)=&\mathcal{N}\left(\fe^{-\rho\alpha\partial_x}\Phi(t_n)-i\rho
G(t_n)+\xi_1^n(\rho)\right)\nonumber\\
=&\mathcal{N}\left(\fe^{-\rho\alpha\partial_x}\Phi(t_n)-i\rho
G(t_n)\right)+\int_0^1\mathcal{N}'\left(\fe^{-\rho\alpha\partial_x}\Phi(t_n)-i\rho
G(t_n)+s\xi_1^n(\rho)\right)\xi_1^n(\rho)ds\nonumber\\
=&\mathcal{N}\left(\fe^{-\rho\alpha\partial_x}\Phi(t_n)\right)
-\rho\mathcal{N}'\left(\fe^{-\rho\alpha\partial_x}\Phi(t_n)\right)(iG(t_n))+\xi_2^n(\rho),
\label{proof2eq2}
\end{align}
where we denote
\begin{align*}
  \xi_2^n(\rho)=&\rho^2\int_0^1\mathcal{N}''\left(\fe^{-\rho\alpha\partial_x}\Phi(t_n)-is\rho
G(t_n)\right)(-iG(t_n),-iG(t_n))(1-s)ds\\
&+\int_0^1\mathcal{N}'\left(\fe^{-\rho\alpha\partial_x}\Phi(t_n)-i\rho
G(t_n)+s\xi_1^n(\rho)\right)\xi_1^n(\rho)ds.
\end{align*}
Then it is direct to see that
$$\|\xi_2^n(\rho)\|_r\leq C\tau^2\|G(t_n)\|_r+C\|\xi_1^n(\rho)\|_r
\leq C\tau^2,\quad 0\leq\rho\leq\tau. $$
By plugging (\ref{proof2eq1}) and (\ref{proof2eq2}) into (\ref{dumammel2}), we find
\begin{align*}
\Phi(t_{n+1})=&\fe^{-\tau\alpha\partial_x}\Phi(t_n)
 -i\int_0^\tau\fe^{-(\tau-\rho)\alpha\partial_x}
\mathcal{L}\left[\fe^{-\rho\alpha\partial_x}\Phi(t_n)-i\rho
G(t_n)+\xi_1^n(\rho)\right]d\rho\\
&-i\int_0^\tau\fe^{-(\tau-\rho)\alpha\partial_x}
\mathcal{N}\left(\fe^{-\rho\alpha\partial_x}\Phi(t_n)-i\rho
G(t_n)+\xi_1^n(\rho)\right)d\rho\\
=&\fe^{-\tau\alpha\partial_x}\Phi(t_n)
 -i\int_0^\tau\fe^{-(\tau-\rho)\alpha\partial_x}
\mathcal{L}\fe^{-\rho\alpha\partial_x}\Phi(t_n)d\rho
-\int_0^\tau\fe^{-(\tau-\rho)\alpha\partial_x}
\mathcal{L}\rho
G(t_n)d\rho\\
&-i\int_0^\tau\fe^{-(\tau-\rho)\alpha\partial_x}
\mathcal{N}\left(\fe^{-\rho\alpha\partial_x}\Phi(t_n)\right)d\rho
+i\int_0^\tau\fe^{-(\tau-\rho)\alpha\partial_x}
\rho\mathcal{N}'\left(\fe^{-\rho\alpha\partial_x}\Phi(t_n)\right)(iG(t_n))d\rho\\
&-i\int_0^\tau\fe^{-(\tau-\rho)\alpha\partial_x}\left[\mathcal{L}\xi_1^n(\rho)
+\xi_2^n(\rho)\right]d\rho,
\end{align*}
which by the definition of the local error gives
\begin{align}
 \xi^n=&-i\int_0^\tau\fe^{-(\tau-\rho)\alpha\partial_x}\left[\mathcal{L}\xi_1^n(\rho)
+\xi_2^n(\rho)\right]d\rho-\int_0^\tau\fe^{-(\tau-\rho)\alpha\partial_x}
\mathcal{L}\rho
G(t_n)d\rho+\frac{\tau^2}{2}\mathcal{L}G(t_n)\nonumber\\
&+i\int_0^\tau\fe^{-(\tau-\rho)\alpha\partial_x}
\rho\mathcal{N}'\left(\fe^{-\rho\alpha\partial_x}\Phi(t_n)\right)(iG(t_n))d\rho
-\frac{i\tau^2}{2}\mathcal{N}'(\Phi(t_n))(iG(t_n)).\label{proof2eq3}
\end{align}
Again by Taylor expansion, we have
\begin{align}\label{proof2eq4}
 -\int_0^\tau\fe^{-(\tau-\rho)\alpha\partial_x}
\mathcal{L}\rho
G(t_n)d\rho+\frac{\tau^2}{2}\mathcal{L}G(t_n)=-\int_0^\tau\rho^2
\int_1^{s/\rho}\fe^{-\rho(1-\kappa)\alpha\partial_x}\alpha\mathcal{L}
\partial_xG(t_n)d\kappa d\rho,
\end{align}
and
\begin{align}
 &i\int_0^\tau\fe^{-(\tau-\rho)\alpha\partial_x}
\rho\mathcal{N}'\left(\fe^{-\rho\alpha\partial_x}\Phi(t_n)\right)(iG(t_n))d\rho
-\frac{i\tau^2}{2}\mathcal{N}'(\Phi(t_n))(iG(t_n))\nonumber\\
=&i\int_0^\tau\tau\int_1^{\rho/\tau}
\fe^{-\tau(1-\kappa)\alpha\partial_x}\alpha\partial_x d\kappa\,\rho
\mathcal{N}'\left(\fe^{-\rho\alpha\partial_x}\Phi(t_n)\right)(iG(t_n))d\rho
+i\int_0^\tau
\rho\xi_3^n(\rho)d\rho,\label{proof2eq5}
\end{align}
where
\begin{align*}\xi_3^n(\rho):=&\mathcal{N}'\left(\fe^{-\rho\alpha\partial_x}\Phi(t_n)\right)(iG(t_n))
-\mathcal{N}'(\Phi(t_n))(iG(t_n))\\
=&\mathcal{N}'\left(\Phi(t_n)-\rho\int_0^1
\fe^{-\sigma\rho\alpha\partial_x}\alpha\partial_x\Phi(t_n)d\sigma\right)(iG(t_n))
-\mathcal{N}'(\Phi(t_n))(iG(t_n))\\
=&\int_0^1\mathcal{N}''\left(\Phi(t_n)-s\rho\int_0^1
\fe^{-\sigma\rho\alpha\partial_x}\alpha\partial_x\Phi(t_n)d\sigma\right)(iG(t_n))
\left(-\rho\int_0^1
\fe^{-\sigma\rho\alpha\partial_x}\alpha\partial_x\Phi(t_n)d\sigma\right)ds.
\end{align*}
It is direct to see
$$\|\xi_3^n(\rho)\|_r\leq C\tau\|\Phi\|_{L^\infty((0,T),H^{r+1})}\leq C\tau,\quad 0\leq \rho\leq\tau.$$
By plugging (\ref{proof2eq4}) and (\ref{proof2eq5}) into (\ref{proof2eq3}) and then taking the $H^r$-norm on both sides, we have
\begin{align*}\|\xi^n\|_r\leq&C\int_0^\tau\left[\|\xi_1^n(\rho)\|_r
+\|\xi_2^n(\rho)\|_r\right]d\rho+C\int_0^\tau \rho^2\|\partial_xG(t_n)\|_rd\rho
+C\int_0^\tau\rho\left[\tau\|G(t_n)\|_r+\|\xi_3^n(\rho)\|_r\right]d\rho\\
\leq& C\tau^3,\quad n\geq0.
\end{align*}

The rest of the proof then proceeds by using the induction for the stability and error propagations which is similar to that of Theorem \ref{main}, and it will be omitted  here for brevity.
\end{proof}

For the second-order ULI scheme \eqref{LI} (or \eqref{LIdp}),  we establish $L^2$ error estimates for data in $H^1$ following the line of argumentation given in \cite{Lubich}.

\begin{corollary}\label{cor:L2}
($L^2$ convergence of extended ULI) Let $\Phi^n$ denote the numerical solution of the second-order  ULI scheme (\ref{LIdp}) for solving the Dirac-Poisson system (\ref{dirac-poisson}) (respectively, of (\ref{LI}) for solving NDE (\ref{dirac trun}) with given $V$).  Let  $\Phi\in L^\infty((0,T);H^{1})$ and $\partial_t\Phi\in L^\infty((0,T);L^2)$
for some $T>0$. Then there exist constants $\tau_0,\,C>0$ depending on $\Vert \Phi\Vert_{L^\infty((0,T);H^{1})}$, $\Vert \partial_t\Phi\Vert_{L^\infty((0,T);L^{2})}$ and $T$, such that for all $0<\tau\leq\tau_0$ and $0\leq t_n\leq T$, we have
 $$\|\Phi(t_n)-\Phi^n\|_{L^2}\leq C\tau^2.$$
\end{corollary}
\begin{proof}
The proof follows by combining the proof of Theorem \ref{main2} with the line of argumentation taken in \cite{Lubich}: Under the regularity assumptions of Corollary \ref{cor:L2} we can prove that the second-order scheme converges with order one in the Sobolev space $H^{\frac12 + \varepsilon}$ (for any $0 < \varepsilon <\frac12 $), i.e., the error $e^n = \Phi(t_n)-\Phi^n$ satisfies
\begin{align}\label{eps}
\| e^n\|_{H^{\frac12 + \varepsilon}}\leq C\tau,\quad 0\leq t_n\leq T.
\end{align}
The error estimate \eqref{eps} implies a priori the boundedness of the numerical solution in $H^{\frac12 + \varepsilon}$  for some $\varepsilon>0$ as
\[
 \|\Phi^n\|_{H^{\frac12 + \varepsilon}} \leq \|e^n\|_{H^{\frac12 + \varepsilon}} +  \| \Phi(t_n) \|_{H^{\frac12 + \varepsilon}}  \leq 2C
 \]
 for $\tau_0$ sufficiently small.

 The a priori boundedness of the numerical solution $\Phi^n$ in the stronger norm $\Vert \cdot \Vert_{H^{\frac12 + \varepsilon}}$ then allows us to apply bilinear estimates in the stability argument with the numerical solution measured in $H^{\frac12 + \varepsilon}$ as
 \[
 \Vert \Phi^n e^n \Vert_{L^2} \leq \Vert \Phi^n\Vert_{L^\infty} \Vert e^n \Vert_{L^2} \leq 2C\Vert e^n \Vert_{L^2} .
 \]
Using classical bilinear estimates based on Sobolev embedding theorem with the numerical solution measured in $H^{\frac12 + \varepsilon}$ allow us to prove $L^2$ stability estimates and obtain second-order convergence in $L^2$ for $H^1$ data.
\end{proof}

\section{Numerical results}\label{sec:result}
In this section, we present the numerical results of the proposed ULI schemes for solving the NDE (\ref{dirac trun}) with external $V=V_e$ and for the Dirac-Poisson system (\ref{dirac-poisson}).
We carry out convergence tests of the ULI schemes, i.e.,  the first-order scheme (\ref{ULI}) (respectively, (\ref{ULIdp})) and the second-order scheme (\ref{LI}) (respectively, (\ref{LIdp})). For reasons of ease we call them \textbf{ULI1} and \textbf{ULI2} in the following. We compare their results with the standard numerical methods that are reviewed in Section \ref{sec:2}. To highlight the advantage of the ULI schemes compared to classical discretization techniques, we focus on the temporal discretization error of these numerical methods. The spatial discretizations of the methods here are all made by the Fourier pseudospectral method \cite{ST} with very fine mesh so that the error is negligible compared to the temporal discretization error.

To construct an initial data $\Phi_0(x)$ for the NDE (\ref{dirac trun}) or (\ref{dirac-poisson}) such that
$\Phi_0\in H^{\theta}(\bT)$ for some specified $\theta\geq0$, we adopt the way from  \cite{lownls}. Choose $N>0$ as an even integer and discrete the spatial domain $\bT$ with grid points
$x_j=j\frac{2\pi}{N}$ for $j=0,\ldots,N$.
Take two uniformly distributed random vectors $\mathrm{rand}(N,1)\in[0,1]^N$ and denote
$$\mathcal{U}^N=\mathrm{rand}(N,1)
+i\,\mathrm{rand}(N,1).$$
Then we define
\begin{equation}\label{non-smooth}
\Phi_0(x):=\frac{|\partial_{x,N}|^{-\theta}\mathcal{U}^N}
{\||\partial_{x,N}|^{-\theta}\mathcal{U}^N\|_{L^\infty}},\quad
x\in\bT,
\end{equation}
where the pseudo-differential operator $|\partial_{x,N}|^{-\theta}$ reads: for Fourier modes $k=-N/2,\ldots, N/2-1$,
\begin{equation*}
 \left(|\partial_{x,N}|^{-\theta}\right)
 _k=\left\{\begin{split}
 &|k|^{-\theta}\quad \mbox{if}\ k\neq0,\\
  &0\qquad\ \ \mbox{if}\ k=0.
  \end{split}\right.
\end{equation*}
We fix $\lambda=1$ in the Thirring type nonlinearity and fix the external potential (if considered) as
$$V_e(x)=2\sin(x),\qquad x\in\bT.$$
In our following numerical experiments, $N$ is fixed as $N=2^{15}$ which is large enough to neglect the spatial error, and the reference solutions
are obtained numerically by the Strang splitting method with very small step size, e.g. $\tau=10^{-6}$ and $N=2^{15}$. We solve the NDE (\ref{dirac trun}) with the given $V=V_e$ or the Dirac-Poisson system (\ref{dirac-poisson}) by the numerical methods under different time step $\tau$ and we compute their relative error in computing the solution at $T=1$.

To test the first-order methods, we construct initial data $\Phi_0(x)\in H^{2.4}(\bT)$ as in (\ref{non-smooth}), and the convergence results of ULI1 (\ref{ULI}), FD1 (\ref{SIFD1}), Lie splitting (\ref{TSFP1}) and EI1 (\ref{EI1}) for solving NDE (\ref{dirac trun}) with external $V=V_e$ are presented in Figure  \ref{fig2}. We measure the error in $H^2$-norm so that the chosen initial data fails the convergence requirements of all the standard methods. The corresponding convergence results for solving the Dirac-Poisson system (\ref{dirac-poisson}) are also presented in Figure \ref{fig2}.

\begin{figure}[t!]
$$\begin{array}{cc}
\psfig{figure=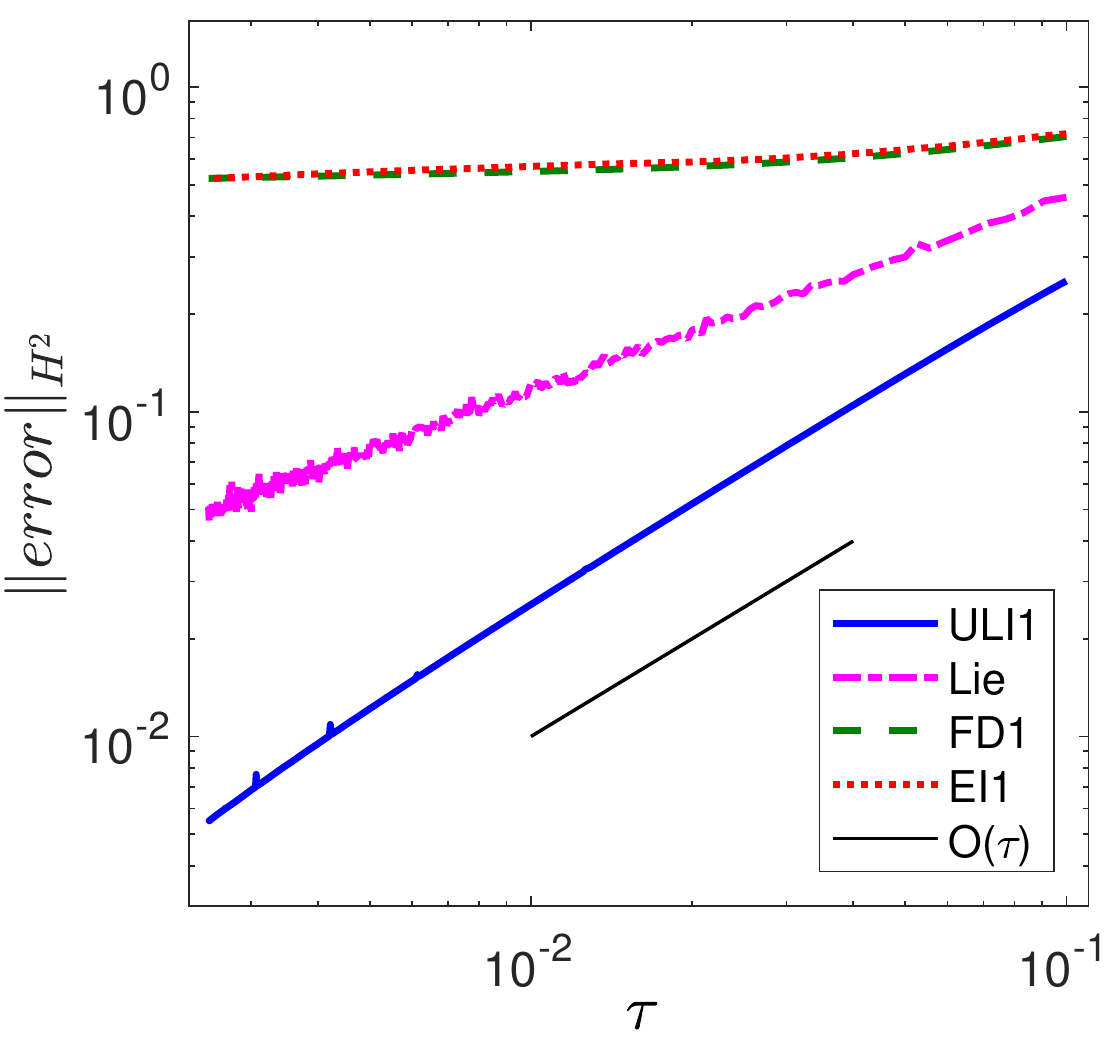,height=6cm,width=6.5cm}&
\psfig{figure=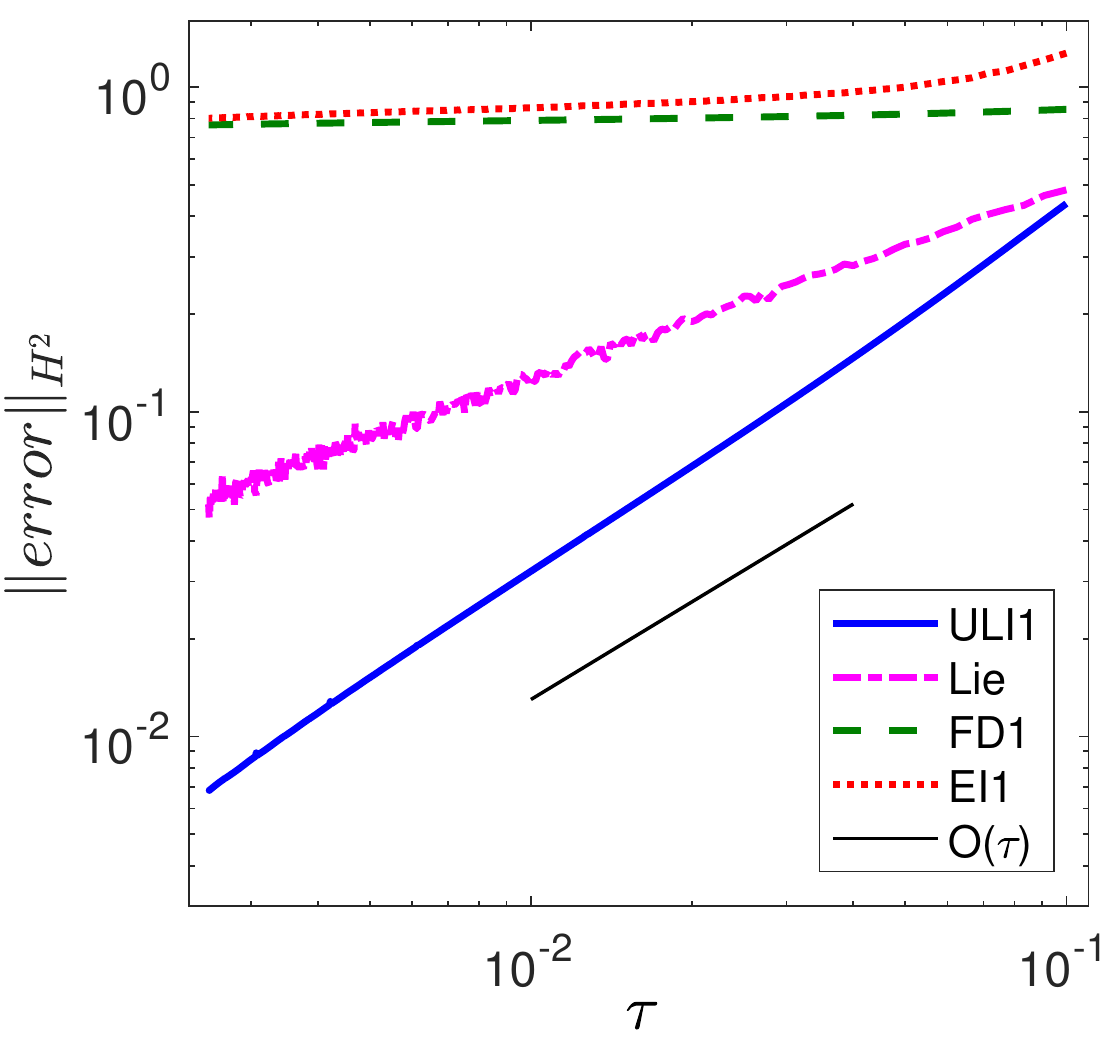,height=6cm,width=6.5cm}
\end{array}$$
\caption{Convergence of the first-order methods for NDE with external $V=V_e$ (left) and for Dirac-Poisson system (right) under  $H^{2.4}$-initial data:
$error=(\Phi(t_n)-\Phi^n)/\|\Phi(t_n)\|_{H^2}$ for $t_n=T=1$.}
\label{fig2}
\end{figure}

For the second-order methods, i.e., ULI2 (\ref{LI}), FD2 (\ref{SIFD2}), Strang splitting (\ref{TSFP2}) and EI2 (\ref{EI2}), their convergence results for solving the NDE (\ref{dirac trun}) with external $V=V_e$ under $\Phi_0(x)\in H^{2.2}(\bT)$ are presented in Figure \ref{fig3}, and the corresponding results for the Dirac-Poisson system (\ref{dirac-poisson}) are shown in Figure \ref{fig3} as well. Here, we measure the error in $H^1$-norm so that the chosen initial data fails the critical regularity requirement of the standard methods.

Last but not least, we would like to test the convergence of the ULI2 when the initial data fails the critical regularity requirement in Theorem \ref{main2}. We take $\Phi_0(x)\in H^{1.4}(\bT)$ and apply the second-order methods for solving the NDE (\ref{dirac trun}) with external $V=V_e$ or the Dirac-Poisson system (\ref{dirac-poisson}). The error measured in $H^1$-norm is shown in Figure \ref{fig4}.

\begin{figure}[t!]
$$\begin{array}{cc}
\psfig{figure=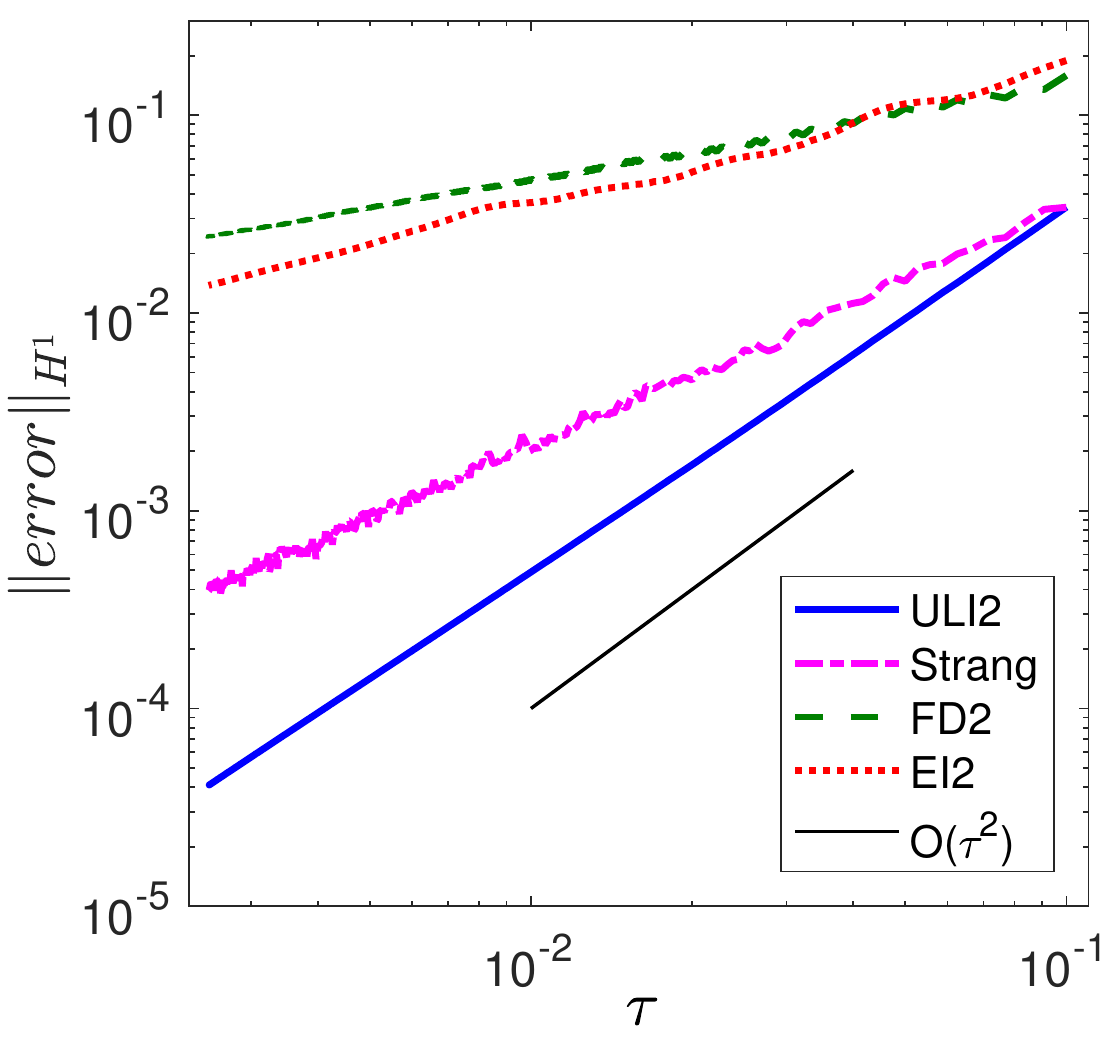,height=6cm,width=6.5cm}&
\psfig{figure=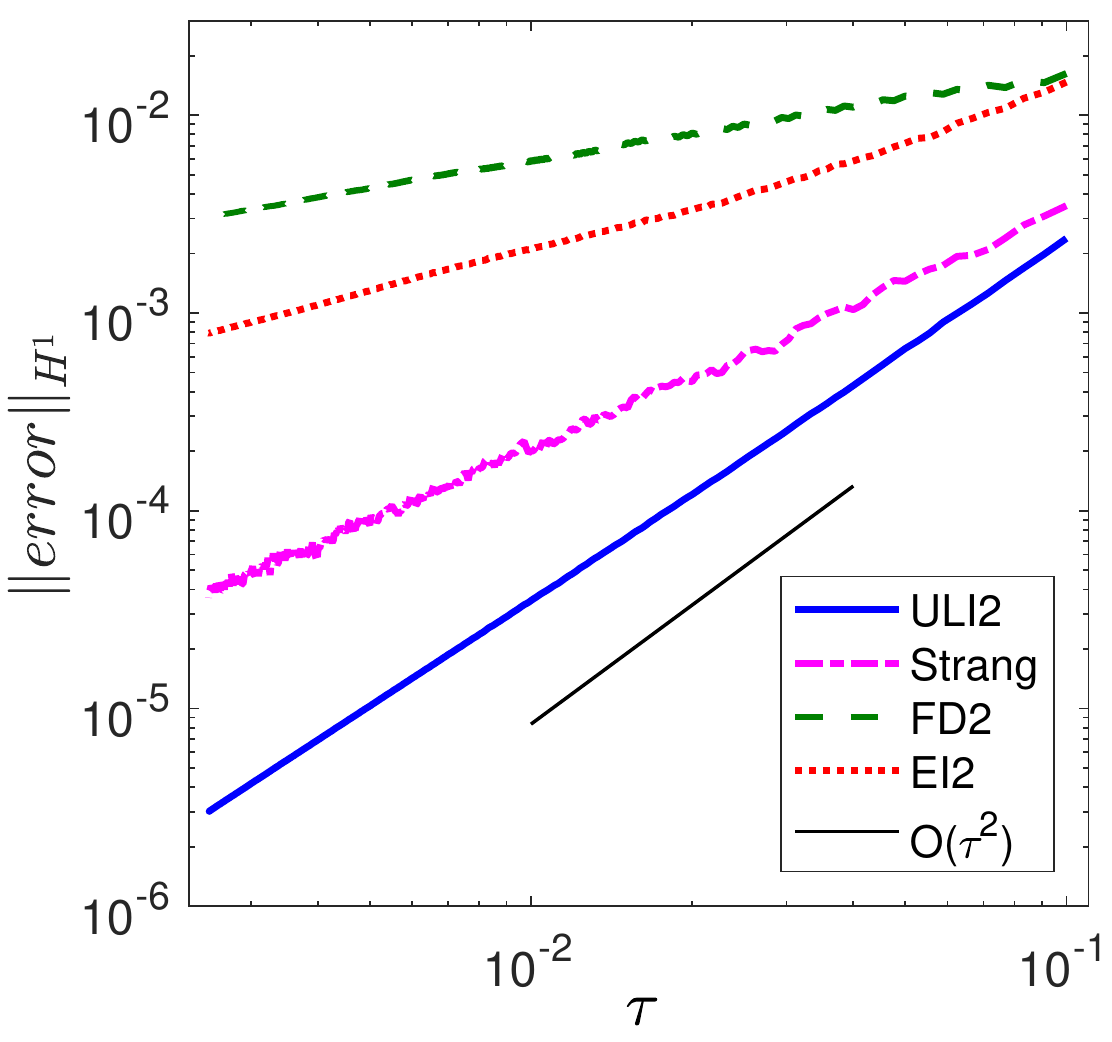,height=6cm,width=6.5cm}
\end{array}$$
\caption{Convergence of the second-order methods for NDE with external $V=V_e$ (left) and for Dirac-Poisson system (right) under $H^{2.2}$-initial data:
$error=(\Phi(t_n)-\Phi^n)/\|\Phi(t_n)\|_{H^1}$ for $t_n=T=1$.}
\label{fig3}
\end{figure}

\begin{figure}[t!]
$$\begin{array}{cc}
\psfig{figure=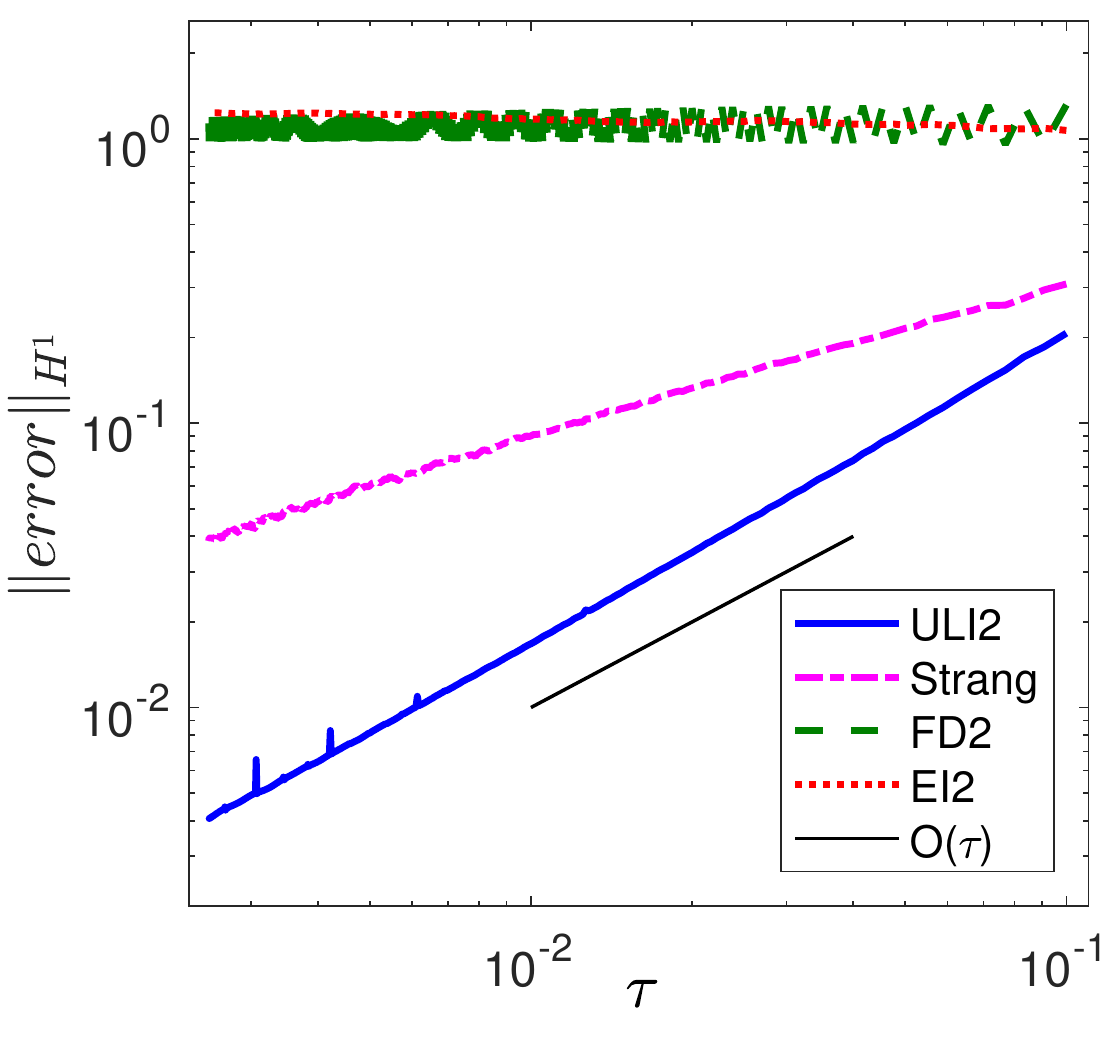,height=6cm,width=6.5cm}&
\psfig{figure=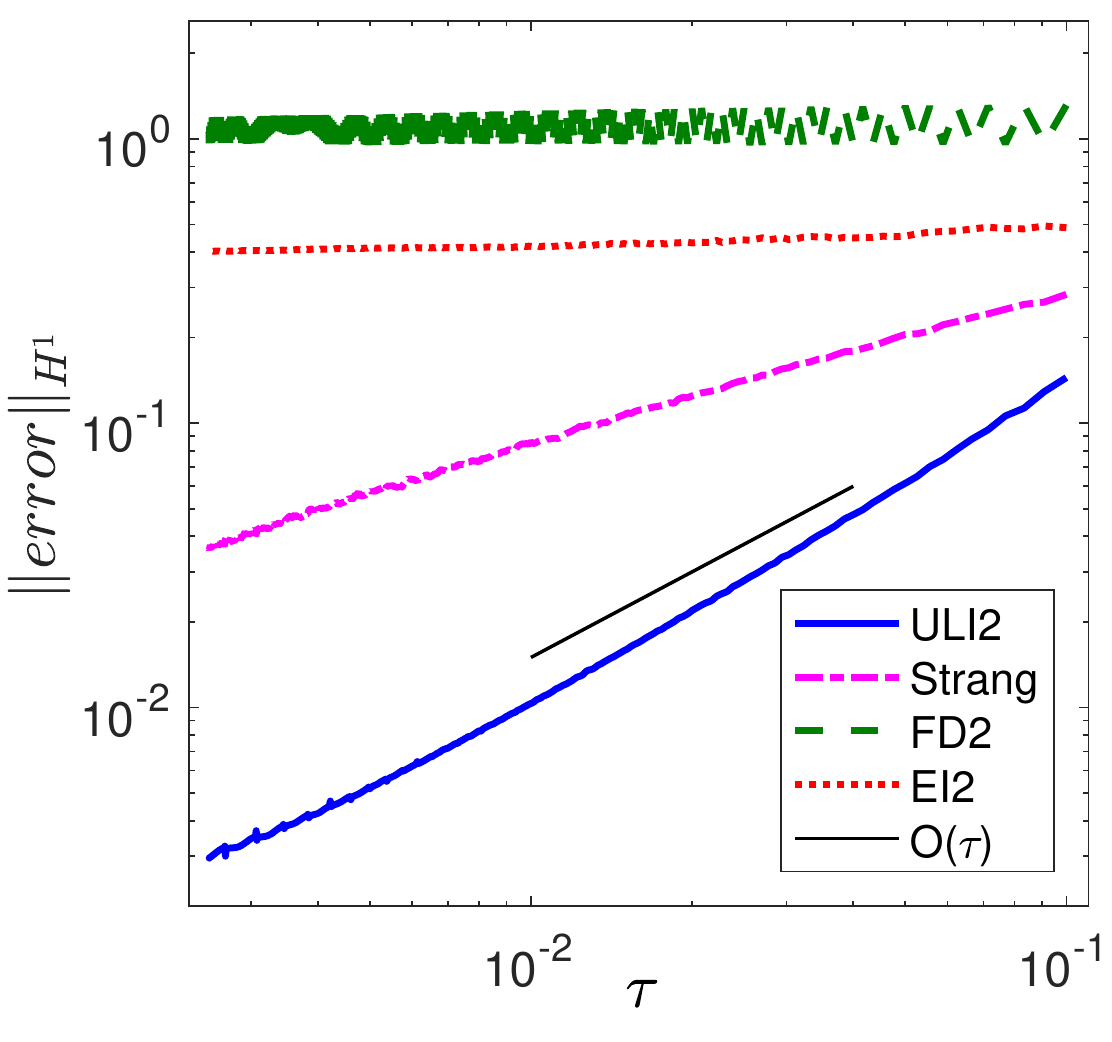,height=6cm,width=6.5cm}
\end{array}$$
\caption{Convergence of the second-order methods for NDE with external $V=V_e$ (left) and for Dirac-Poisson system (right) under $H^{1.4}$-initial data:
$error=(\Phi(t_n)-\Phi^n)/\|\Phi(t_n)\|_{H^1}$ for $t_n=T=1$.}
\label{fig4}
\end{figure}

From our numerical results (Figures \ref{fig2}-\ref{fig4}), we can draw the following conclusions:

1) All the tested methods are free from any CFL conditions (note that we used $\tau\gg \Delta x\approx 4\times 10^{-4}$). However, the standard numerical methods from Section \ref{sec:2} all suffer from significant order reduction for rough initial data. In particular, we numerically observe the favorable error behavior of our new schemes at low regularity, which underlines our theoretical convergence results.

2) The proposed ULI schemes reach their optimal convergence rates for rough initial data, which verifies our error estimates in Theorem \ref{main} and Theorem \ref{main2}. Therefore, the ULI methods are more accurate and efficient than the standard methods for solving the NDEs with less regular solutions.

3) ULI2 shows first-order convergence rate in $H^1$-norm for $H^{1+\delta}$-initial data with $0<\delta<1$. The latter does not meet the regularity requirement in Theorem \ref{main2}. Nevertheless,  ULI2 shows a much improved convergence rate compared to classical schemes.

\section{Conclusion}\label{sec:con}
In this paper, we consider numerical methods for integrating the nonlinear Dirac equation and the Dirac-Poisson system (NDEs) with low regular solutions. Due to the numerical loss of derivatives, standard methods such as finite difference methods, classical exponential integrators and splitting methods all suffer from   order reduction if the solution does not satisfy the critical regularity requirement. We propose a new class of ultra low-regularity integrators (ULI) for solving the NDEs. The great advantage of the new schemes is that they do not require any additional smoothness of the solution, i.e., ULI offers optimal first-order convergence rate in $H^r$ for solutions in $H^r$. Rigorous convergence results are established, and the extension of ULI to a  second-order scheme is established. Numerical experiments confirm our theoretical error estimates and underline the favorable error behavior of the new schemes  at low regularity.

\appendix
\section{Proof of convergence theorems for standard methods}\label{appendix}
Here we provide very briefly the proofs of the convergence theorems for the standard methods in Section \ref{sec:2} to emphasize the loss of derivatives in the classical approximations. We consider only the first-order convergence results under the external $V$ case for simplicity, and the second-order results can be proved by similar analysis, where the techniques are borrowed largely from the existing work \cite{SplittingDirac1,Lubich,lownls}. We denote in the following the error function as
$$e^n:=\Phi(t_n)-\Phi^n,\quad n\geq0. $$

\subsection{Proof of Theorem \ref{thm FD}}
\begin{proof}
  Let $\xi^n$ be the local truncation error defined as
  \begin{align}\label{xi fd}
  \xi^n=\frac{i}{\tau}\left[\Phi(t_{n+1})-\Phi(t_n)\right]+i\alpha\partial_x\Phi(t_{n+1})
  -\beta\Phi(t_n)-V\Phi(t_n)-F(\Phi(t_n)),\quad n\geq0.
  \end{align}
  By the Taylor expansion and the equation itself (\ref{dirac trun a}), we find
 \begin{align*}
  \xi^n=i\tau\int_0^1\partial_{tt}\Phi(t_n+\tau\sigma)(1-\sigma)d\sigma
  +i\tau\alpha\int_0^1\partial_t\partial_x\Phi(t_{n}+\tau\sigma)d\sigma,\quad n\geq0.
  \end{align*}
Then under the assumption
  $\partial_t\Phi\in L^\infty((0,T);H^{r+1})$ and $\partial_{tt}\Phi\in L^\infty((0,T);H^{r})$, we have
  $$\|\xi^n\|_r\leq C\tau\left(\|\partial_{tt}\Phi\|_{L^\infty((0,T);H^r)}+
  \|\partial_{t}\Phi\|_{L^\infty((0,T);H^{r+1})} \right)\leq C\tau.$$
  By taking the difference between (\ref{xi fd}) and the scheme (\ref{SIFD1}), we get
  $$\frac{i}{\tau}\left(e^{n+1}-e^n\right)=-i\alpha\partial_xe^{n+1}
  +\eta^n+\xi^n,\quad n\geq0,$$
  where $\eta^n:=\beta e^n+Ve^n+F(\Phi(t_n))-F(\Phi^n)$. Thus, we have
  $$e^{n+1}=A_\tau e^n
  -i\tau A_\tau(\eta^n+\xi^n),\quad\mbox{with}\quad A_\tau:=(id+\tau\alpha\partial_x)^{-1}
  =\frac{1}{1-\tau\partial_{xx}}\begin{pmatrix}
                                  1 &-\tau\partial_x\\
                                  -\tau\partial_x &1
                                \end{pmatrix}.$$
   For some general $\Psi=(\psi_1,\psi_2)^T\in H^r(\bT)$, by direct computing we see that the Fourier coefficients satisfy
   $$   |\widehat{(A_\tau\Psi)}_l|^2=|\widehat{\Psi}_l|^2\frac{1+\tau^2l^2}{(1+\tau^2l^2)^2},
   \quad l\in\bZ,$$
  so we have $\|A_\tau \Psi\|_r\leq \|\Psi\|_r.$
Then the rest of proof follows in the induction manner (see in the proof of Theorem \ref{main}) with the help of bilinear estimates.
\end{proof}

\subsection{Proof of Theorem \ref{thm EI}}
\begin{proof}
  The local truncation error of the EI1 scheme reads
  \begin{equation}
  \label{local EI}\xi^n=\Phi(t_{n+1})-\fe^{-i\tau\mathcal{T}}\Phi(t_n)
  +i\tau\varphi_1(-i\tau\mathcal{T})G(\Phi(t_n)),
  \quad n\geq0.
  \end{equation}
  By the Duhamel's formula (\ref{duhammel0}) and Taylor expansion, we find
  \begin{align*}
  \xi^n=-i\int_0^\tau\fe^{-i(\tau-s)\mathcal{T}}\int_0^1
  sG'(\Phi(t_n+s\sigma))\partial_t\Phi(t_n+s\sigma)d\sigma ds.
  \end{align*}
  So under the assumption in Theorem \ref{thm EI} and by the bilinear estimates, we get
  $$\|\xi^n\|_r\leq C\tau^2\|\partial_t\Phi\|_{L^\infty((0,T);H^{r})}\leq C\tau^2,
  \quad n\geq0.$$
  Then by taking the difference between (\ref{locaerror1}) and (\ref{EI1}), we get
  $$e^{n+1}=\fe^{-i\tau\mathcal{T}}e^n-
 i\tau\varphi_1(-i\tau\mathcal{T})\left[G(\Phi(t_n))-G(\Phi^n)\right]
 +\xi^n,\quad n\geq0.$$
 It is direct to verify that $\fe^{-i\tau\mathcal{T}}$ is isometric in $H^r(\bT)$, i.e.,  \begin{equation}\label{T isometric}
 \|\fe^{-i\tau\mathcal{T}}\Psi\|_r
 =\|\Psi\|_r,\quad \forall \Psi\in H^r(\bT),
 \end{equation}
 and
 $$\|\varphi_1(-i\tau\mathcal{T})\Psi\|_r\leq C\|\Psi\|_r.
$$
Hence, we have
\begin{align*}
 \|e^{n+1}\|_r\leq\|e^n\|_r+C\tau\|G(\Phi(t_n))-G(\Phi^n)\|_r+\|\xi^n\|_r,\quad n\geq0,
\end{align*}
and the rest of proof follows by the induction manner.
\end{proof}

\subsection{Proof of Theorem \ref{thm SP}}
\begin{proof}
Plugging (\ref{TSFP1a}) into (\ref{TSFP1b}),  the local truncation error is defined as
\begin{equation}\label{local TSFP}
\xi^n=\Phi(t_{n+1})-\fe^{-i\tau\mathcal{T}}\fe^{-i\tau
B^n}\Phi(t_n),\quad n\geq0,
\end{equation}
where we denote $B^n=V\cdot Id+\lambda(|\phi_1(t_n)|^2-|\phi_2(t_n)|^2)\beta$.
By Taylor expansion, we have
\begin{align*}
\fe^{-i\tau
B^n}\Phi(t_n)
=\Phi(t_n)-i\tau B^n\Phi(t_n)
+\xi^n_1\quad \mbox{with}\quad \xi_1^n=-\tau^2(B^n)^2\int_0^1\fe^{-is\tau
B^n}(1-s)ds\,\Phi(t_n),\end{align*}
and then
\begin{equation}\label{TSFP eq1}
\xi^n=\Phi(t_{n+1})-\fe^{-i\tau\mathcal{T}}
\Phi(t_n)+i\tau\fe^{-i\tau\mathcal{T}}B^n\Phi(t_n)
-\fe^{-i\tau\mathcal{T}}\xi^n_1.
\end{equation}
By iterating the Duhamel's formula (\ref{duhammel0}) once, i.e., the Picard iteration, we get
$$\Phi(t_{n+1})
=\fe^{-i\tau\mathcal{T}}\Phi(t_n)
 -i\int_0^\tau\fe^{-i(\tau-s)\mathcal{T}}G\left(
 \fe^{-is\mathcal{T}}\Phi(t_n)+\zeta_1^n(s)\right)\, ds,$$
 where
 $$\zeta_1^n(s):=-i\int_0^s\fe^{-i(s-\rho)\mathcal{T}}G\left(
 \Phi(t_{n}+\sigma)\right)d\sigma,\quad 0\leq s\leq\tau.$$
 Then by Taylor expansion
 \begin{align}\label{TSFP eq2}
  \Phi(t_{n+1})
=&\fe^{-i\tau\mathcal{T}}\Phi(t_n)
 -i\int_0^\tau\fe^{-i(\tau-s)\mathcal{T}}G\left(
 \fe^{-is\mathcal{T}}\Phi(t_n)\right)\, ds
 +\xi_2^n,
 \end{align}
where
$$\xi_2^n=-i\int_0^\tau\fe^{-i(\tau-s)\mathcal{T}}\int_0^1G'\left(
 \fe^{-is\mathcal{T}}\Phi(t_n)+\sigma\zeta_1^n(s)\right)\zeta_1^n(s)d\sigma ds.$$
 Plugging (\ref{TSFP eq2}) into (\ref{TSFP eq1}) and noting $G(\Phi(t_n))=B^n\Phi(t_n)$, we get
\begin{align*}
 \xi^n=&i\tau\fe^{-i\tau\mathcal{T}}G(\Phi(t_n))
 -i\int_0^\tau\fe^{-i(\tau-s)\mathcal{T}}
 G\left(\fe^{-is\mathcal{T}}\Phi(t_n)\right)\, ds
 -\fe^{-i\tau\mathcal{T}}\xi^n_1+\xi_2^n.
\end{align*}
By Taylor expansion, we have
\begin{align*}
\zeta_2^n(s):=& \fe^{is\mathcal{T}}
 G\left(\fe^{-is\mathcal{T}}\Phi(t_n)\right)-G(\Phi(t_n))\\
=&is\mathcal{T}\int_0^1\fe^{is\sigma\mathcal{T}}d\sigma\,
 G\left(\fe^{-is\mathcal{T}}\Phi(t_n)\right)
 -\int_0^1 G'\left(\Phi(t_n)
 -\sigma\zeta_3^n(s) \right)\zeta_3^n(s)d\sigma
\end{align*}
where
$$\zeta_3^n(s)=i\int_0^1 s\mathcal{T}\fe^{-is\rho\mathcal{T}}\Phi(t_n)d\rho.$$
Therefore,
\begin{align*}
 \xi^n=&-i\fe^{-i\tau\mathcal{T}}\int_0^\tau\zeta_2^n(s)ds
 -\fe^{-i\tau\mathcal{T}}\xi^n_1+\xi_2^n.
\end{align*}
Based on the assumption and the bilinear estimates, we find
\begin{align*}
&\|\xi_1^n\|_r\leq C\tau^2\|\Phi\|_{L^\infty((0,T);H^{r})}\leq C\tau^2,\\
&\|\xi_2^n\|_r\leq C\int_0^\tau\|\zeta_1^n(s)\|_rds\leq \tau^2\|\Phi\|_{L^\infty((0,T);H^{r})}\leq C\tau^2,\quad n\geq0,\\
 &\|\zeta_2^n(s)\|_r\leq sC\|\Phi\|_{L^\infty((0,T);H^{r+1})}
  +\int_0^1\|\zeta_3^n(s)\|_rds\leq sC\|\Phi\|_{L^\infty((0,T);H^{r+1})}\leq C\tau,\quad
  0\leq s\leq\tau.
\end{align*}
Hence, we have
\begin{align*}
 \|\xi^n\|_r&\leq\int_0^\tau\|\zeta_2^n(s)\|_rds+\|\xi_1^n\|_r+\|\xi_2^n\|_r \leq C\tau^2,\quad n\geq0.
\end{align*}
Taking the difference between (\ref{local TSFP}) and (\ref{TSFP1}), we get
$$e^{n+1}=\fe^{-i\tau\mathcal{T}}
\left[\fe^{-i\tau
B^n}\Phi(t_n)-\fe^{-i\tau(V\cdot Id+\lambda(|\phi_1^n|^2-|\phi_2^n|^2)\beta)}\Phi^n\right],
\quad n\geq0.$$
Then thanks to the fact (\ref{T isometric}), we have
$$\|e^{n+1}\|_r=\left\|\fe^{-i\tau\left(V\cdot Id+\lambda(|\phi_1(t_n)|^2-|\phi_2(t_n)|^2)\beta\right)}\Phi(t_n)
-\fe^{-i\tau\left(V\cdot Id+\lambda(|\phi_1^n|^2-|\phi_2^n|^2)\beta\right)}\Phi^n\right\|_r,\quad n\geq0,$$
and the rest of the proof proceeds similarly as that in \cite{Lubich} for stability and convergence.
\end{proof}

\section*{Acknowledgements}
Y. Wang is supported by the Fundamental Research Funds for the Central Universities CCNU19TD010. X. Zhao acknowledges the starting research grant from Wuhan University.

\bibliographystyle{model1-num-names}

\end{document}